\documentclass[12pt,a4paper,reqno]{amsart}
\usepackage{amssymb,pictexwd,dcpic,hyperref}
\usepackage[all]{xy}
\theoremstyle{plain}
\newtheorem{thrm}{Theorem}[section]
\newtheorem*{thrm*}{Theorem}

\newtheorem*{prop*}{Proposition}
\newtheorem{lemma}[thrm]{Lemma}
\newtheorem*{lemma*}{Lemma}

\newtheorem*{cor*}{Corollary}
\theoremstyle{remark}

\newtheorem*{note}{Note}

\theoremstyle{remark}
\newtheorem*{defn}{Definition}

\newtheorem*{example}{Example}
\newtheorem*{examples}{Examples}
\newtheorem{itprop}[thrm]{Proposition}
\newtheorem*{itprop*}{Proposition}

\DeclareMathOperator{\im}{im}                                % image
\newcommand{\N}{\mathbb N}                                   % Natural numbers
\newcommand{\Z}{\mathbb Z}                                   % Integers
\newcommand{\R}{\mathbb R}                                   % Reals
\newcommand{\C}{\mathbb C}                                   % Complex numbers
\newcommand{\set}[1]{\left\{#1\right\}}                      % set
\newcommand{\abs}[1]{\left\lvert#1\right\rvert}              % absolute value
\newcommand{\norm}[1]{\left\lVert#1\right\rVert}             % norm
\newcommand{\imp}{\Rightarrow}                               % implies (short)
\newcommand{\tox}[1]{\xrightarrow{#1}}                       % to arrow with superscript
\theoremstyle{remark}
\newtheorem*{motiv}{Motivation}
\newcommand{\B}{\mathcal B}                                  % B (bounded operators)
\newcommand{\D}{\mathbb D}                                   % D_q (quantum disc)
\newcommand{\Hi}{\mathcal H}                                 % H (Hilbert space)
\newcommand{\K}{\mathcal K}                                  % K (compact operators)
\newcommand{\M}{\mathcal M}                                  % M (multiplier algebra)
\newcommand{\T}{\mathcal T}                                  % T (Toeplitz algebra)
\newcommand{\F}{\mathcal F}
\newcommand{\cstalg}{$C\sp *$\nobreakdash-algebra}           % C*-algebra
\newcommand{\stalg}{$*$\nobreakdash-algebra}                 % *-algebra
\newcommand{\cstsubalg}{$C^*$\nobreakdash-subalgebra}        % C*-subalgebra
\newcommand{\stideal}{$*$\nobreakdash-ideal}                 % *-ideal
\newcommand{\sthomo}{$*$\nobreakdash-homomorphism}           % *-homomorphism
\newcommand{\strepn}{$*$\nobreakdash-representation}         % *-representation
\newcommand{\Efam}{$E$\nobreakdash-family}                   % E-family
\newcommand{\clspan}{\overline{\operatorname{span}}}         % closed span
\DeclareMathOperator{\id}{id}                                % id
\DeclareMathOperator{\Irr}{Irr}                              % irreducible representations
\DeclareMathOperator{\Int}{Int}                              % interior
\DeclareMathOperator{\Aut}{Aut}
\DeclareMathOperator{\Ph}{Ph}
\newcommand{\bigbrace}[2]{\begin{cases}#1\\#2\end{cases}}
\newcommand{\casesif}[4]{\begin{cases}#1&\text{if }#2\\#3&\text{if }#4\end{cases}}

\newcommand{\zcasesif}[4]{\casesif{\zeta_{#1}}{#2}{\zeta_{#3}}{#4}}
\newcommand{\casesother}[3]{\begin{cases}#1&\text{if }#2,\\#3&\text{otherwise}\end{cases}}
\newcommand{\casesotherd}[3]{\begin{cases}#1&\text{if }#2,\\#3&\text{otherwise.}\end{cases}}
\newcommand{\zmapsto}[2]{\zeta_{#1}\mapsto\zeta_{#2}}         % z_a |-> z_b
\newcommand{\xmapsto}[1]{\overset{#1}{\mapsto}}
\newcommand{\phase}[1]{#1\abs{#1}^{-1}}                       % x|x|^{-1}
\newcommand{\phaseE}[1]{#1(#1^*#1)^{-1/2}}                    % x(x^*x)^{-1/2}
\newcommand{\qrel}[1]{#1^*#1-q#1#1^*}                         % x^*x-qxx^*
\newcommand{\projrel}[1]{#1(#1^*#1)^{-1}#1^*}                 % x(x^*x)^{-1}x^*

\newcommand{\V}{\textbf{.}}  % vertex
\newcommand{\graphloop}{\xygraph{[d(0.2)]{u}="u"(:^{a}@(ul,ur)"u")}}  % one vertex, one edge + labels
\newcommand{\graphloopS}{\xygraph{[d(0.2)]{\V}="v"(:^{}@(ul,ur)"v")}}  % one vertex, one edge
\newcommand{\graphEtwoS}{\xygraph{[d(0.2)]{\V}="v"(:^{}@(ul,ur)"v":^{}[r]{\V}="w")}}  % smaller, label-free version
\newcommand{\graphSthree}{\xygraph{[d(0.2)]{u}="u"(:^{a}@(ul,ur)"u":^{c}[r]{v}="v":^{b}@(ul,ur)"v")}}
\newcommand{\graphSthreeS}{\xygraph{[d(0.2)]{\V}="u"(:^{}@(ul,ur)"u":^{}[r]{\V}="v":^{}@(ul,ur)"v")}}
\newcommand{\graphEtwo}{\xygraph{[d(0.2)]{v}="v"(:^{b}@(ul,ur)"v":^{e}[r(2)]{w}="w")}}  % v,w, b,e
\newcommand{\graphEfour}{\xygraph{[d(0.2)]{u}="u"(:^a@(ul,ur)"u":^c[r(2)]{v}="v":^b@(ul,ur)"v":^e[r(2)]{w}="w", :_d@/_1pc/"w")}}
\newcommand{\graphEfourv}{\xygraph{[d(0.2)]{u}="u"(:^a@(ul,ur)"u":^c[r]{v}="v":^b@(ul,ur)"v":^e[r]{w}="w", :_d@/_0.7pc/"w")}}
\newcommand{\graphEfourS}{\xygraph{[d(0.2)]{\V}="u"(:^{}@(ul,ur)"u":^{}[r]{\V}="v":^{}@(ul,ur)"v":^{}[r]{\V}="w", :_{}@/_0.6pc/"w")}}
\newcommand{\graphStwoS}{\xygraph{{\V}="v"(:^{}@(ul,ur)"v":^{}[l(0.9)]{\V}="w",:^{}[r(0.9)]{\V}="x")}}
\newcommand{\morphgraph}[2]{\xygraph{{#1}="v"(:^-\phi@/^1.3pc/[r(2.3)]{#2}="w":^-\psi@/^1.3pc/"v")}}

\begin{document}

\title{Quantum Even-Dimensional Balls}
%\date{\today}
\date{February 2005}
\author{Colin MacLaurin}
%\email{colin.maclaurin@studentmail.newcastle.edu.au}
%\email{colin.maclaurin@uqconnect.edu.au}

\begin{abstract}
The quantum even-dimensional balls are defined as the \cstalg s generated by certain graphs. We exhibit a polynomial algebra for each even-dimensional quantum ball, and classify the irreducible representations of it. 
\end{abstract}

\maketitle

\footnotesize
\emph{Later note}: This honours thesis in mathematics was submitted in February 2005, to complete a BMath(Hons) degree at the University of Newcastle, Australia. It was supervised by Prof. Wojciech Szyma\'nski (then Senior Lecturer), and has since been cited in the paper [HS2]. Thank-you to Wojciech for an excellent project topic, and generosity.
% citation mentioned in email from Wojciech on 11th August 2011

Extended summary: For the quantum disc ($2$-ball), it was already known the following $C^*$-algebras are isomorphic:
\[ C(\D_q) \quad\cong\quad \T \quad\cong\quad C^*\Bigl(\graphEtwoS\Bigr) \vspace{-6mm} \]
These are respectively: the $C^*$-algebra of the quantum disc, the Toeplitz algebra, and a graph $C^*$-algebra. The present work extends this to every quantum $2n$-ball, showing isomorphism between: the completion of a polynomial algebra generated by certain ``quantum'' relations (Equation~\ref{eqn:quantumrelations}), a universal algebra generated by unilateral-shift-like operators, and a graph $C^*$-algebra. See the Appendices for a detailed summary. \qquad --- CM\footnote{Email: \emph{colin.maclaurin@uqconnect.edu.au}}, \emph{12th November 2019}

\normalsize

\tableofcontents
\newpage

\section{Preliminaries}

\subsection{\texorpdfstring{$C^*$}{Cstar}-algebras}

\subsubsection{Preliminaries}

A \emph{\cstalg} is a Banach \stalg\ $A$ which satisfies the \emph{$C^*$-identity}
\[ \norm{a^*a}=\norm{a}^2\, \forall a\in A. \]
For example, the set $B(\Hi)$ of bounded operators on a Hilbert space is a \cstalg, as is the set $C(X)$ of continuous $\C$-valued functions on a compact Hausdorff topological space $X$.

The \emph{Gelfand theorem} states that the category of Hausdorff topological spaces with continuous functions is dual to the category of abelian unital \cstalg s with \sthomo s between them. 

\subsubsection{Representations}

A \emph{\strepn} of a \cstalg\ $A$ on a Hilbert space $\Hi$ is a \sthomo\ $\pi\colon A\to\B(\Hi)$. We say $\pi$ is \emph{nondegenerate} if
\[ \set{\pi(a)h\colon a\in A, h\in\Hi} \]
is dense in $\Hi$. By representation we will mean a nondegenerate \strepn.

We say $\pi$ is \emph{faithful} if $\ker(\pi)=\set{0}$. Also $\pi$ is \emph{irreducible} if there are no nontrivial invariant subspaces of $\Hi$, which is a stronger condition than non-degeneracy. Suppose $\pi\colon A\to\B(\Hi_1)$ and $\rho\colon A\to\B(\Hi_2)$ are two representations of $A$. We say $\pi$ is \emph{unitarily equivalent} to $\rho$ if there exists a unitary operator $u\colon\Hi_1\to\Hi_2$ s.t. $\rho(a)u=u\pi(a)\, \forall a\in A$.

Suppose we have an exact sequence $0\to J\to A\tox{\pi}B\to 0$ where $J$ is an ideal of $A$. Every representation of $J$ or $B$ extends to a representation of $A$ as follows. Suppose $\rho\colon B\to\B(\Hi)$ is a representation of $B$. Then we can compose with the quotient map $\pi$ to get a representation $\rho\circ\pi\colon A\to\B(\Hi)$ of $A$. On the other hand, suppose $\rho\colon J\to\B(H_i)$ is a representation of $J$. Then $\set{\rho(x)h\colon x\in J, h\in H}$ is dense in $\Hi$, so any element of $\Hi$ can be approximated by $\rho(x)h$. We can extend $\rho$ to a representation of $A$ by defining
\[ \rho(a)\left(\rho(x)h\right) := \rho(ax). \]
Moreover, irreducible representations extend, by this procedure, to irreducible ones. It is a theorem that these are all the irreducible representations of $A$:
\[ \Irr(A) = \Irr(J) \sqcup \Irr(B), \]
where $\sqcup$ means disjoint union.
If the ideal $J$ is essential, then a faithful representation on $J$ extends to a faithful representation on $A$. A theorem due to Gelfand-Naimark states that every \cstalg\ has a faithful nondegenerate representation.

\subsubsection{Projections and Isometries}

A \emph{projection} in a \cstalg\ is an element $p$ satisfying $p=p^*=p^2$. In any representation $\pi$, $\pi(p)$ is the orthogonal projection onto a subspace. We say an element $v$ is an \emph{isometry} if $v^*v=1$. In any representation $\pi$, $\norm{\pi(v)h}=\norm{h}\, \forall h$. A \emph{proper isometry} is an isometry which is not unitary.

A \emph{partial isometry} is an element $s$ such that $s^*s$ is a projection. Then $ss^*$ is also a projection and $s=ss^*s$. We say $s$ has \emph{domain projection} $s^*s$ and \emph{range projection} $ss^*$. A \emph{partial unitary} is a partial isometry for which the domain and range projections coincide. 

\subsubsection{Matrix Units and $\K(\Hi)$}

A \emph{system of matrix units} indexed by a set $X$ is a collection $\set{u_{ij}\colon i,j\in X}$ of linearly independent partial isometries in a \cstalg\ which satisfy the relations

\begin{itemize}
\item $u_{ij}u_{mn} = \casesif{u_{in}}{j=m,}{0}{j\ne m.}$
\item $u_{ij}^*=u_{ji}$.
\end{itemize}

Let $\Hi$ be an infinite-dimensional separable Hilbert space, and $R$ be the set of finite-rank operators on $\Hi$. Then the set of \emph{compact operators} $\K=\K(\Hi)$ is defined to be the norm-closure of $R$. $\K$ has only one irreducible representation up to unitary equivalence.

It can be shown that the closed span of a countably infinite system of matrix units in a \cstalg\ is isomorphic to $\K$.

\subsubsection{Polar Decomposition}

Recall that we can write any complex number $a+bi$ in polar form $re^{i\theta}$, where $r\in\R$ is positive and $\abs{e^{i\theta}}=1$. There is an analogous procedure for operators on a Hilbert space:

\begin{itprop} 
Every operator $T\in\B(\Hi)$ has a \emph{polar decomposition}
\[ T=V\abs{T}, \]
where $\abs{T}:=(T^*T)^{1/2}$ is positive and $V$ is a partial isometry. If we require that $V^*V=(ker T)^\perp$ for the domain projection of $V$, then this decomposition is unique. If $T$ is invertible then $V$ is unitary.
\end{itprop}

Note that $T^*T$ is positive, so its spectrum is contained in $[0,\infty)$. So we can apply the functional calculus to the function $\sqrt{\cdot}$ which is well defined on $[0,\infty)$. However we can't always find a polar decomposition of an element in a general \cstalg. A sufficient condition that an element $x$ of a unital \cstalg\ $A$ has a polar decomposition is that $\abs{x}$ is invertible, in which case
\[ x=u\abs{x}, \qquad\text{where } u:=\phase{x}. \]
More generally, suppose that $x=pxp$ where $p\in A$ is a projection. Then $x$ belongs to the \cstsubalg\ $pAp$, which is a unital \cstalg\ with unit $p$. If $\abs{x}$ is invertible in $pAp$, then $x$ has a polar decomposition $x=\alpha\abs{x}$ where the phase $\alpha$ satisfies $\alpha^*\alpha=p$. In the examples we will see, we will also have $\alpha\alpha^*<p$.

When an element $x$ of a unital \cstalg\ has a polar decomposition $x=u\abs{x}$ we say $\abs{x}$ is the \emph{modulus} of $x$ and $u$ is the \emph{phase} of $x$, by analogy with the complex numbers.

\subsubsection{The Busby invariant}

The Busby invariant combines most of the information in an extension $0\to A\to B\to C\to 0$ into a single morphism.

A \emph{unitisation} of a \cstalg\ $A$ is a unital \cstalg\ in which $A$ can be embedded as an essential ideal. The smallest unitisation of $A$ is written $A^\sim$, and the largest unitisation is called the \emph{multiplier algebra} $\M(A)$. For example, $\B(\Hi)$ is the multiplier algebra for $\K(\Hi)$.

The \emph{corona algebra} of $A$ is the \cstalg\ $\M(A)/A$. The corona algebra is a generalisation of the Calkin algebra $\B/\K$.

The \emph{Busby invariant} for an extension $E\colon 0\to A\tox{\alpha}B\tox{\beta}C\to 0$ is the morphism $\tau_E \colon C\to\M(A)/A$ defined by
\[ \tau_E(c):=\pi\circ\sigma(b), \]
where $b\in B$ is a lift of $c$ through $\beta$, $\pi$ is the quotient map $\colon\M(A)\to\M(A)/A$, and $\sigma$ is the unique \sthomo\ $\colon B\to\M(A)$ which makes the following diagram commutative:
\[
\begindc{\commdiag}
\obj(10,30)[A]{$A$}
\obj(30,30)[B]{$B$}
\obj(30,10)[C]{$M(A)$}
\mor{A}{B}{$\alpha$}
\mor{A}{C}{$\iota$}
\mor{B}{C}{$\sigma$}
\enddc
\]

\subsubsection{Universal $C^*$-algebras} \cite{B}

A universal \cstalg\ is defined on a set $G$ of generators and a set $R$ of relations. A representation of $(G,R)$ is a set of operators on a Hilbert space $\Hi$ which satisfy the relations. We say that $(G,R)$ is \emph{admissible} if there exists a representation, and the representations sum together nicely. If $(G,R)$ is admissible then
\[ \norm{z}_0 := \sup\set{\norm{\rho(z)}\colon \rho \text{ a representation of $(G,R)$}} \]
is a well-defined finite number, and $\norm{\cdot}_0$ is a $C^*$-semiform. Then $J_0 := \set{z\colon \norm{z}_0=0}$ is a two-sided ideal of $\F(G)$, and the completion of the quotient $\F(G)/J_0$ under $\norm{\cdot}_0$ is called the \emph{universal \cstalg\ of $(G,R)$}. 

It has the property that there is a surjective homomorphism from the universal algebra to any given \cstalg\ with generators which satisfy the relations.

\begin{examples}
The quantum disc is the universal \cstalg\ generated by the relation $\qrel{x}=(1-q)1$. Also a graph algebra is a universal \cstalg\ with generators and relations determined from a graph. We will see both of these examples shortly. \end{examples}

Universal algebras have convenient properties for showing isomorphism. Suppose we have two universal \cstalg s $A$ and $B$. It is enough to find maps $\phi\colon A\to B$ and $\psi\colon B\to A$ defined on generators which preserve the relations (so they extend to homomorphisms), composition gives the identity i.e. $\phi\circ\psi=\id_B$, $\psi\circ\phi=\id_A$, and such that the images of the generators are nonzero. Then $\phi$ and $\psi$ are isomorphisms.

\subsection{Graph Algebras} \cite{GrphAlgs}

\subsubsection{Definitions}
A \emph{directed graph} $E=(E^0,E^1,r,s)$ consists of countable sets $E^0$ and $E^1$ of \emph{vertices} and \emph{edges} respectively, and maps $r,s\colon E^1\to E^0$ describing the \emph{range} and \emph{source} of the edges. We represent a graph by putting points for vertices and arrows for the edges connecting the vertices.

\begin{example}
$E^0=\set{v,w}$, $E^1=\set{b,e}$ in the following graph:
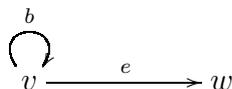
\[ \graphEtwo  \vspace{-6mm} \]
Then $r(b)=v$, $r(e)=w$ and $s(b)=s(e)=v$.
\end{example}

A \emph{sink} is a vertex which emits no edges, i.e. $s^{-1}(v)=\emptyset$, whereas a \emph{source} is a vertex which receives no edges, i.e. $r^{-1}(v)=\emptyset$. We will consider \emph{row-finite graphs}, in which every vertex emits at most finitely many edges.
             
\begin{defn}
Let $E$ be a row-finite directed graph. A \emph{Cuntz-Krieger \Efam} consists of families of projections $\set{P_v\colon v\in E^0}$ and partial isometries $\set{S_e\colon e\in E^1}$ on some Hilbert space $\Hi$ s.t.
\begin{enumerate}
\item [(G1)] $P_vP_w=0$ for $v,w\in E^0$ with $v\ne w$
\item [(G2)] $S_e^*S_e=P_{r(e)}$\, for $e\in E^1$
\item [(G3)] $P_v=\sum_{e\in E^1\colon s(e)=v} S_eS_e^*$\, for $v\in E^0$, provided $v$ is not a sink.
\end{enumerate}
\end{defn}

The relation (G1) says that the projections are mutually orthogonal. The relation (G3) is known as the Cuntz-Krieger relation.

\begin{itprop}
For any row-finite directed graph $E$, there is a \cstalg\ $C^*(E)$ generated by a Cuntz-Krieger \Efam\ $\set{P_v,S_e}$ s.t. for every other Cuntz-Krieger \Efam\ $\set{Q_v,R_e}$ in a \cstalg\ $B$, there is a surjective \sthomo\ $\pi\colon C^*(E)$ \newline
$\to B$ s.t. $\pi(P_v)=Q_v$ and $\pi(S_e)=R_e$ for all $v\in E^0$, $e\in E^1$.
\end{itprop}

In other words, $C^*(E)$ is the \emph{universal} \cstalg\ for these relations. $C^*(E)$ is called the \emph{\cstalg\ of $E$}, and is described as a \emph{graph \cstalg}. It is also possible to generalise to the case of graphs which are not row-finite \cite{FLR}.

There is a continuous action $\gamma\colon S^1\to\Aut(C^*(E))$ called the \emph{gauge action} which satisfies $\gamma_t(s_e)=ts_e$, $\gamma_t(p_v)=p_v$. We say an ideal $J$ of $C^*(E)$ is \emph{gauge-invariant} if $\gamma_tJ=J$ for all $t\in S^1$. Existence of the gauge action is equivalent to universality of the \cstalg\ $C^*(E)$.

\subsubsection{Paths in $E$}
A \emph{path} in $E$ is a sequence of edges $\mu=(\mu_1,\dotsc,\mu_n)$ s.t. $r(\mu_i)=s(\mu_{i+1})$ for all $i$. $E^*$ is the set of all finite paths in $E$, including the zero-length paths, which are just vertices. If we define $S_\mu:=S_{\mu_1}S_{\mu_2}\dotsb S_{\mu_n}$, we see that $S_\mu\ne 0$ only if $\mu$ is a path, in which case $S_\mu$ is a partial isometry.

\begin{itprop}
$C^*(E)=\clspan\set{S_\mu S_\nu^*\colon \mu,\nu\in E^*,\, r(\mu)=r(\nu)}$.
\end{itprop}

\subsubsection{Ideals}
For $v,w\in E^0$, we write $v\ge w$ if there is a path from $v$ to $w$. We say a subset $H$ of $E^0$ is
\begin{itemize}
\item \emph{hereditary} if $w\in H$, $w\ge v$ $\imp$ $v\in H$,
\item \emph{saturated} if for $v\in E^0$ not a sink, $\set{r(e)\colon s(e)=v}\subset H$ $\imp$ $v\in H$.
\end{itemize}

For a row-finite graph $E$ the gauge-invariant ideals in $C^*(E)$ are in a one-to-one correspondence with the saturated hereditary sets of $E^0$. If $I$ is an ideal in $C^*(E)$ then the set $H_I:=\set{v\in E^0\colon P_v\in I}$ is hereditary and saturated. Conversely, if $H\subset E^0$ is hereditary and saturated, then the ideal $I_H$ generated by $\set{P_v\colon v\in H}$ is gauge-invariant.

Furthermore, the quotient $C^*(E)/I_H$ is the graph algebra $C^*(E\backslash H)$, where $I_H$ is the ideal corresponding to some hereditary and saturated set $H$. $E\backslash H$ is the graph obtained by removing all vertices in $H$ and all edges with source or range in $H$.

\subsection{Toeplitz Algebra}

\subsubsection{Definition}

Recall $\ell^2(\N)$ is the Hilbert space of square-summable sequences, where it will be assumed that $\N$ includes $0$. Let $\set{\zeta_0,\zeta_1,\dotsc}$ be the standard orthonormal basis of $\ell^2(\N)$. The \emph{shift operator} or \emph{unilateral shift} $S$ is defined on $\ell^2(\N)$ by
\[ S\zeta_i:=\zeta_{i+1}. \]
$S$ is a bounded linear operator with adjoint
\[ S^*\zeta_i = \casesif{0}{i=0}{\zeta_{i-1}}{i\ge 1.} \]
The shift operator is a proper isometry, since $S^*S=1$ but $SS^*\ne 1$.

\begin{defn}
The \emph{Toeplitz algebra} is the \cstalg\ generated by the unilateral shift:
\[ \T := C^*(S). \]
\end{defn}

\noindent \emph{Coburn's Uniqueness Theorem} states that if $u\in A$, $v\in B$ are proper isometries in \cstalg s $A$ and $B$ then there is an isomorphism $\phi\colon C^*(u)\to C^*(v)$ s.t. $\phi(u)=v$. In particular, any \cstalg\ generated by a proper isometry is isomorphic to $\T$.

$\K$ is an ideal of $\T$ generated by $1-SS^*$, and $\T$ fits into the short exact sequence
\[ 0\to\K\to T\to C(S^1)\to 0. \]

\subsection{Quantum Disc} \cite{KL,PDS,QQS,C-M}

\begin{defn}
The \emph{polynomial algebra} of the \emph{quantum disc} is the \stalg\
\[ P(\D_q):=\C\langle x,x^* \rangle/J_q, \qquad 0<q<1, \]
where $\C\langle x,x^* \rangle$ is the unital free \stalg\ generated by $x$ and $x^*$, and $J_q$ is the \stideal\ generated by the relation
\begin{equation} \label{Dqeqn}
x^*x-qxx^*=(1-q)1.
\end{equation}
\end{defn}

Every \strepn\ $\pi\colon P(\D_q)\to\B(\Hi)$ satisfies $\norm{\pi(x)}=1$, since
\[ \norm{\pi(x)}^2 = \norm{\pi(x)^*\pi(x)} = \norm{q\pi(x)\pi(x)^*+(1-q)1} = q\norm{\pi(x)}^2+(1-q), \]
so $\norm{\pi(x)}^2=1$ and hence $\norm{\pi(x)}=1$. Hence $\sup_\pi\norm{\pi(a)}<\infty$ for all $a\in P(\D_q)$, and so we can define
\[ \norm{a} := \sup_\pi\norm{\pi(a)}. \]
This is a $C^*$-algebraic norm, but it is not complete.

\begin{defn}
The \emph{\cstalg} $C(\D_q)$ of the \emph{quantum disc}, for $0<q<1$, is the $C^*$-completion of $P(\D_q)$.
\end{defn}

This completion is the universal \cstalg\ generated by the relation $x^*x-qxx^*=(1-q)1$.

\begin{itprop} \label{repnqdisc}
Every irreducible \strepn\ of $C(\D_q)$ is unitarily equivalent to one of the following representations:
\begin{enumerate}
\item a one-dimensional representation $\rho_\theta$ defined by $\rho_\theta(x)=e^{i\theta}$, $\rho_\theta(x^*)=e^{-i\theta}$, for $0\le\theta<2\pi$, or
\item an infinite-dimensional representation $\pi_q$ defined on a Hilbert space with orthonormal basis $\set{e_i}_{i\in\N}$ by
\begin{align*}
\pi_q(x)e_i &= \sqrt{1-q^{i+1}} e_{i+1}, \quad i\ge 0, \\
\pi_q(x^*)e_i &= \begin{cases} 0 & i=0, \\ \sqrt{1-q^i} e_{i-1} & i\ge 1. \end{cases}
\end{align*}
\end{enumerate}
\end{itprop}

The infinite-dimensional representation $\pi_q$ is faithful. The one-dimensional representations correspond to the classical points, forming a circle, of the quantum disc. If we think of $C(S^1)$ as the \cstalg\ generated by $a$ and $a^*$ subject to the relations $aa^*=a^*a=1$, then this classical circle can be embedded into the quantum disc via the homomorphism $\phi_q\colon C(\D_q)\to C(S^1)$ defined by $x\mapsto a$.

There are also unbounded representations of $C(\D_q)$. For $0<q<1$, $C(\D_q)$ is isomorphic to the Toeplitz algebra $\T$, as shown next.

\subsubsection{Isomorphism to the Toeplitz Algebra}\label{TcongDq}

Define an operator $z$ on $\ell^2(\N)$ by:
\begin{equation} \label{zToeplitz}
z := \sum_{k=0}^\infty \left(\sqrt{1-q^{k+1}}-\sqrt{1-q^k}\right) S^{k+1}(S^*)^k,
\end{equation}
with the convention $(S^*)^0=1$. This sum converges since the sequence of partial sums $S_n$ is Cauchy. For $n>m$,
\[ \norm{S_n-S_m} = \norm{\sum_{k=m+1}^n\lambda_kS^{k+1}(S^*)^k} \le \sum_{k=m+1}^n\lambda_k = \sqrt{1-q^{n+1}}-\sqrt{1-q^{m+1}}, \]
where $\lambda_k:=\sqrt{1-q^{k+1}}-\sqrt{1-q^k}$. As $n,m\to\infty$, $\sqrt{1-q^{n+1}}-\sqrt{1-q^{m+1}}\to 0$, and so the sum converges. Clearly $z\in\T$. To see how $z$ acts on the basis elements $\zeta_i$, we evaluate
\begin{align*}
(S^*)^k\zeta_i &= \casesif{0}{i<k}{\zeta_{i-k}}{i\ge k,} \\
S^{k+1}\zeta_i &= \zeta_{i+k+1},
\end{align*}
and so
\[ S^{k+1}(S^*)^k\zeta_i = \casesif{0}{i<k}{\zeta_{i+1}}{i\ge k.} \]
Hence in the sum defining $z(\zeta_i)$, the terms are all zero for $k>i$. Hence
\[ z(\zeta_i) = \sum_{k=0}^i \lambda_k S^{k+1}(S^*)^k\zeta_i = \sqrt{1-q^{i+1}}\zeta_{i+1}, \]
since $\sum_{k=0}^i \lambda_k=\sqrt{1-q^{i+1}}$. So $z$ is a \emph{weighted shift} and has adjoint
\[ z^*(\zeta_i) = \casesif{0}{i=0,}{\sqrt{1-q^i}\zeta_{i-1}}{i\ge 1.} \]
Thus one may show that
\[ z^*z-qzz^*=(1-q)1, \]
which is the defining relation for the quantum disc. Also note that we can recover $S$ from $z$ by considering the polar decomposition of $z$. The operator $\abs{z}=(z^*z)^{1/2}$ satisfies $\abs{z}(\zeta_i)=\sqrt{1-q^{i+1}}\zeta_i$ which is invertible, so $z$ has a unique polar decomposition with phase
\[ \phase{z}\colon\zeta_i\xmapsto{\abs{z}^{-1}}(1-q^{i+1})^{-1/2}\zeta_i\xmapsto{z}\zeta_{i+1}, \]
which is just $S$. In particular we can recover $S$ from $z$, and so $C^*(S)=C^*(z)$.

To show $\T\cong C(\D_q)$, define maps
\[ \morphgraph{C^*(S)}{C(\D_q)} \]
by $\phi(S):=\phase{x}$ and $\psi(x):=z$ on generators. To show these maps extend to homomorphisms we must show they preserve the relations. We have already shown that $\psi$ preserves the quantum relation. To see that $\phase{x}$ is a partial isometry, consider the faithful representation $\pi$ of $C(\D_q)$ above. By Proposition \ref{repnqdisc}, $\pi(x)e_i=\sqrt{1-q^{i+1}}e_{i+1}$, so by the same working as for $z$ we have $\pi(\phase{x})e_i=e_{i+1}$ and this operator is a partial isometry.

Now $\psi\circ\phi=\id$ since $(\psi\circ\phi)(S) = \psi(\phase{x}) = \phase{\psi(x)} = \phase{z} = S$. Also $(\phi\circ\psi)(x) = \phi(z) = \phi(\sum_{k=0}^\infty \lambda_k S^{k+1}(S^*)^k) = \sum_{k=0}^\infty \lambda_k (\phase{x})^{k+1}(\phase{x})^{*k} = x$, which can be seen by computing in the representation $\pi$ and repeating calculations similar to those above. Hence $\phi\circ\psi=\id_{C(\D_q)}$ and $\psi\circ\phi=\id_{\T}$, which implies that $\phi$ and $\psi$ are injective and surjective and hence isomorphisms. This demonstrates the convenience of using universal properties to show isomorphism.

We will use a similar procedure in the case of $C(B_q^4)$, the quantum $4$-ball, to show the isomorphism of the graph algebra defining $C(B_q^4)$ and the completion of the polynomial algebra for it. 

\subsubsection{Isomorphism with a Graph Algebra}
\label{sec:isomorphisms}

Consider the following graph $E$:
\[ \graphEtwo \vspace{-6mm} \]
Then $C^*(E)$ is the universal \cstalg\ generated by projections and partial isometries $\set{P_v,P_w,S_b,S_e}$ subject to the Cuntz-Krieger relations
\[ P_vP_w = 0, \quad S_b^*S_b = P_v, \quad S_e^*S_e = P_w, \quad P_v = S_bS_b^* + S_eS_e^*. \]
Define $S:=S_b + S_e$. One may verify that $S^*S=1$ and $SS^*=P_v<1$, so $S$ is a proper isometry. Moreover $C^*(E)$ is generated by $S$, since we can express the projections and partial isometries in terms of $S$:
\[ P_v = SS^*, \quad P_w = 1-SS^*, \quad S_b = S^2S^*, \quad S_e = S(1-SS^*), \]
which one can verify using the Cuntz-Krieger relations. Hence by Coburn's Theorem $C^*(E)\cong\T$.

The following \cstalg s are all isomorphic:
\[ C(\D_q) \qquad\cong\qquad \T \qquad\cong\qquad C^*\Bigl(\graphEtwoS\Bigr) \vspace{-6mm} \]
giving us a clear understanding of $C(\D_q)$. We now define higher-dimensional quantum balls using what is known as the quantum double suspension, and then find polynomial algebras for, and classify representations of these quantum balls.

\subsection{Quantum Double Suspension} \cite{SPGA}

\begin{defn}
Let $\Omega$ be a compact topological space. The \emph{suspension} $S\Omega$ of $\Omega$ is the quotient of $\Omega\times[0,1]$ with both $\Omega\times\set{0}$ and $\Omega\times\set{1}$ collapsed to single points.
\end{defn}

Applying the suspension twice yields the \emph{double suspension} of a space. Then the \cstalg\ $C(S^2\Omega)$ of continuous functions on $S^2\Omega$ is an essential extension
\[ 0\to C(\Omega)\otimes C_\infty(\R^2)\to C(S^2\Omega)\to C(S^1)\to 0, \]
where $C_\infty(\R^2)$ is the \cstalg\ of continuous functions on $\R^2$, which vanish at infinity.

We wish to define a quantum analogue of this classical double suspension of a space. There is a procedure called \emph{Weyl quantisation} (not described here) which yields a quantum analogue for certain \cstalg s. The quantum analogue of $C_\infty(\R^2)$ is $\K$, but $C(S^1)$ does not have a quantum analogue. Hence we define the quantum double suspension as the \cstalg\ satisfying the following exact sequence, with a condition on the Busby invariant:

\begin{defn}
Let $A$ be a unital \cstalg. The \emph{quantum double suspension} of $A$ is the unital \cstalg\ $S^2A$ for which there exists an essential extension
\[ 0\to A\otimes\K\to S^2A\to C(S^1)\to 0, \]
such that the corresponding Busby invariant $\beta\colon C(S^1)\to\M(A\otimes\K)/A\otimes\K$ sends
\[ \beta\colon z\mapsto\mu(I\otimes T), \]
where $\mu$ is the canonical surjection of $\M(A\otimes\K)$ onto the quotient $M(A\otimes\K)/A\otimes\K$, $T$ is an isometry with range of codimension $1$, and $z$ is the identity function on the circle $S^1$.
\end{defn}

For graph algebras, the quantum double suspension has a simple realisation:

\begin{defn}
Let $E$ be a graph with finitely many vertices $\set{v_1,\dotsc,v_n}$. The \emph{quantum double suspension} of $E$ is the graph $S^2E$ formed by adding a vertex $v_0$ and edges from $v_0$ to every vertex:
\begin{align*}
(S^2E)^0 &= E^0\cup\set{v_0} \\
(S^2E)^1 &= E^1\cup\set{f_0,\dotsc,f_n}
\end{align*}
\end{defn}
where $s(f_i)=v_0$ and $r(f_i)=v_i$.

\begin{itprop}
Suppose $E$ is a graph with a finite number of vertices and $S^2E$ is its quantum double suspension graph. Then $C^*(S^2E)$ is the quantum double suspension \cstalg\ of $C^*(E)$, i.e.
\[ S^2\left(C^*(E)\right)\cong C^*(S^2E). \]
\end{itprop}

In the classical case, applying the double suspension to a geometric object yields new geometric objects, as shown in Table \ref{tblClassical}:

\begin{table}[!ht]
\begin{tabular}{|c|c|c|l|}
\hline
Top. Space $\Omega$ & Picture & Double Suspension & Iterate to Get Classical: \\
\hline
circle & \circle{16} & $S^2\Omega=S^3$ & $S^2(S^{2n-1})=S^{2n+1}$, odd-dim. spheres \\
two points & $\centerdot\quad\centerdot$ & $S^2\Omega=S^2$ & $S^2(S^{2n})=S^{2(n+1)}$, even-dim. spheres \\
single point & $\centerdot$ & $S^2\Omega=\D$ & $S^2(B^{2n})=B^{2(n+1)}$, even-dim. balls \\
\hline
\end{tabular}
\caption{Classical Case} \label{tblClassical}
\end{table}

Similarly, in the quantum case we can get higher dimensional quantum objects by applying the quantum double suspension, as shown in Table \ref{tblQuantum}:

\begin{table}[!ht]
\begin{tabular}{|c|c|c|c|l|}
\hline
Graph $E$ & $C^*(E)$ & $S^2E$ & $C^*(S^2E)$ & Iterate to Get Quantum: \\
\hline
\graphloopS & $C(S^1)$ & \graphSthreeS & $C(S_q^3)$ & $C(S_q^{2n+1})$, odd-dim. spheres \vspace{-9mm} \\
\textbf{.}\quad\textbf{.} & $C(\centerdot\;\centerdot)$ & \graphStwoS & $C(S_{qc}^2)$, $c\ne 0$ & $C(S_q^{2n})$, even-dim. spheres \\
\textbf{.} & $C(\centerdot)$ & \graphEtwoS & $C(\D_q)\cong\T$ & $C(B_q^{2n})$, even-dim. balls \\
\hline
\end{tabular}
\caption{Quantum Case} \label{tblQuantum}
\end{table}

As suggested in the table, we make the following definition:
\begin{defn}\cite{SPGA}
The \cstalg\ $C(B_q^{2n})$ of continuous functions on the quantum $2n$-ball is the \cstalg\ of the graph obtained by applying the quantum double suspension $n$ times to the single-vertex graph.
\end{defn}

\section{Quantum 4-ball \texorpdfstring{$C(B_q^4)$}{CBq4}} \label{secBq4}

The \cstalg\ $C(B_q^4)$ of continuous functions on the quantum $4$-ball is defined to be the \cstalg\ of the following graph $E_2$, with labels as shown.
\[ \graphEfour \]
The Cuntz-Krieger relations are
\begin{center}
$P_uP_v=0, \quad P_uP_w=0, \quad P_vP_w=0$ \smallskip \\
$S_a^*S_a=P_u, \quad S_b^*S_b=P_v, \quad S_c^*S_c=P_v, \quad S_d^*S_d=P_w, \quad S_e^*S_e=P_w$ \smallskip \\
$P_u=S_aS_a^*+S_cS_c^*+S_dS_d^*, \quad P_v=S_bS_b^*+S_eS_e^*$.
\end{center}

\subsection{Polynomial Algebra}

\begin{defn}
Consider the free \stalg\ $\F(x_1,x_2)$ generated by $x_1$ and $x_2$, and all representations $\pi'$ of this algebra into bounded operators on a Hilbert space which satisfy the relations
\begin{align}
x_1x_2&=0 \\
x_1^*x_2&=0 \\
\qrel{x_2}&=(1-q)1.
\end{align}
In all such representations, the element $\pi'(x_2^*x_2)$ is invertible since
\[ \pi'(x_2^*x_2) = q\pi'(x_2x_2^*)+(1-q)1 > (1-q)1. \]
Hence it makes sense to consider the smaller set of representations $\pi$ which in addition satisfy the relation
\begin{equation}
\qrel{x_1} = (1-q)\left(1-\projrel{x_2}\right).
\end{equation}

Notice that $P:=1-\projrel{x_2}$ is a projection since $1-P = x_2\abs{x_2}^{-2}x_2^*$ is: $(x_2\abs{x_2}^{-2}x_2^*)^* = x_2\abs{x_2}^{-2}x_2^*$ and $(x_2\abs{x_2}^{-2}x_2^*)^2 = x_2\abs{x_2}^{-2}\abs{x_2}^2\abs{x_2}^{-2}x_2^* = x_2\abs{x_2}^{-2}x_2^*$. Also notice that $x_1=Px_1P$, so $x_1$ is in the \emph{corner} in any representation. We have $P\ne 0$ since $x_1P=x_1$, and $x_1\ne 0$ since the graph algebra $C^*(E_2)$ provides a model of these relations in which $x_1\ne 0$, which in turn has the faithful representation $\pi$. We will see that $P$ corresponds to the projection $P_v+P_w\in C^*(E_2)$.

In every representation $\pi$ we have $\norm{\pi(x_1)}=\norm{\pi(x_2)}=1$, as shown for $P(\D_q)$. Hence we can define a semi-norm $\norm{\cdot}_0$ on $\F(x_1,x_2)$ by
\[ \norm{a}_0 := \sup\set{\norm{\pi(a)}}. \]
Let
\[ J_0 := \set{a\in\F(x_1,x_2)\colon \norm{a}_0=0}. \]
It is easy to verify this is a two-sided \stideal\ of $\F(x_1,x_2)$. Define the \emph{polynomial algebra} $P(B_q^4)$ of the quantum $4$-ball as the quotient \stalg\ $\F(x_1,x_2)/J_0$.
\end{defn}

\begin{thrm} \label{thrmE2}
The completion $C^*(x_1,x_2)$ of the polynomial algebra $P(B_q^4)$ is isomorphic to the quantum $4$-ball $C(B_q^4)$.
\end{thrm}

\begin{note}
We are not defining $C(B_q^4)$ as the $C^*$-completion of $P(B_q^4)$, as $C(B_q^4)$ is defined to be the graph algebra $C^*(E_2)$. Rather, we prove that they are isomorphic.
\end{note}

\begin{motiv}
We will see that $C^*(E_2)$ is generated by two elements $S_1$ and $S_2$ defined by
\[ S_1:=S_b+S_e, \qquad S_2:=S_a+S_c+S_d. \]
These elements are like the unilateral shift for the Toeplitz algebra seen earlier. In particular, $C^*(S_1)\cong C^*(S_2)\cong \T$, as we shall see. However $S_1$ and $S_2$ are not convenient. Instead we exhibit elements $z_1$ and $z_2$, like $x$ for the Toeplitz algebra, which also generate $C^*(E_2)$ and can be thought of as polynomial functions. 
Now $C(B_q^4)$ fits into the following exact sequence
\[ 0\to\T\otimes\K\to C(B_q^4)\to C(S^1)\to 0 \]
by the definition of the quantum double suspension. Take a system of matrix units $\set{E_{ij}}$ for $\K$, so $\K=\clspan\set{E_{ij}}$. Then $\T\otimes E_{11}\cong\T$, so (roughly speaking) choose an element $z_1\in\T$ like the element $x$ in the proof that $C(\D_q)\cong\T$. Then $C^*(z_1\otimes E_{11})\cong \T\otimes E_{11}$. Choose another element $z_2\in C(B_q^4)$ also similar to $x$. Then $C^*(z_2)\cong\T$ also.

The ideal structure of the graph algebra suggests the definition of $S_1$, $S_2$ given above. Hence we define
\[ z_1 := \sum_{k=0}^\infty \lambda_k S_1^{k+1}(S_1^*)^k, \qquad z_2 := \sum_{k=0}^\infty \lambda_k S_2^{k+1}(S_2^*)^k, \]
where $\lambda_k:=\sqrt{1-q^{k+1}}-\sqrt{1-q^k}$ as before, and find relations between these elements.
\end{motiv}

\begin{lemma} \label{lemB4a}
$C^*(E_2)=C^*(S_1,S_2)$, in other words, $C^*(E_2)$ is generated by $S_1$ and $S_2$.
\end{lemma}

\begin{proof}
One may verify using the Cuntz-Krieger relations that we can express the projections and partial isometries of $C^*(E_2)$ in terms of $S_1$ and $S_2$:
\begin{align*}
P_u &= S_2S_2^* \\
P_v &= S_1S_1^* \\
P_w &= S_1^*S_1-S_1S_1^* \\
S_a &= S_2^2S_2^* \\
S_b &= S_1^2S_1^* \\
S_c &= S_2S_1S_1^* \\
S_d &= S_2(S_1^*S_1-S_1S_1^*) \\
S_e &= S_1(S_1^*S_1-S_1S_1^*) = S_1(1-S_1S_1^*).
\end{align*}
Hence $C^*(E_2)=C^*(S_1,S_2)$.
\end{proof}

Furthermore, $S_2^*S_2=1$ and $S_2S_2^*=P_u$ so $S_2$ is a proper isometry, with $C^*(S_2)\cong\T$. Also $S_1$ is a partial isometry with $S_1S_1^*=P_v<P_v+P_w=S_1^*S_1$. One may easily verify that $S_1$ belongs to the corner $(P_v+P_w)C^*(E_2)(P_v+P_w)$ which is a unital \cstsubalg\ with unit $P_v+P_w$. Hence $C^*(S_1)$ is contained in this subalgebra, and we have $C^*(S_1)\cong\T$ since $S_1$ is a proper isometry in the subalgebra.

\subsubsection{A faithful representation of $C(B_q^4)$}
We can simplify computations with $z_1$ and $z_2$ by calculating in a faithful representation $\pi$. The set $\set{w}\subset E_2^0$ is hereditary and saturated, with corresponding ideal $J_w$ generated by $P_w$. We define $\pi$ on $J_w$ and extend it to $C^*(E_2)$.

\begin{lemma}
$\set{S_\alpha S_\beta^*\colon \alpha,\beta\in E_2^*,\, r(\alpha)=r(\beta)=w}$ is a system of matrix units with closed span $J_w$. In particular, $J_w\cong\K$.
\end{lemma}

\begin{proof}
Let $\alpha\in E_2^*$. Then $S_\alpha P_w=S_\alpha$ if $r(\alpha)=w$, or $0$ otherwise. Hence $S_\alpha\in J_w$ if $r(\alpha)=w$. So $S_\alpha S_\beta^*\in J_w$ if $r(\alpha)=r(\beta)=w$. Recall that $C^*(E_2)=\clspan\set{S_\gamma S_\delta^*\colon \gamma,\delta\in E_2^*}$, and multiplication by one of these elements yields
\[ S_\alpha S_\beta^* S_\gamma S_\delta^* = \casesif{S_\alpha S_\delta^*}{\gamma=\beta}{S_\alpha S_{\delta\beta'}^*}{\beta\text{ extends }\gamma,\text{ i.e. }\beta=\gamma\beta'} \]
for $r(\alpha)=r(\beta)=w$. Note that $\gamma$ cannot extend $\beta$. The terms are zero unless $r(\delta)=w$, and hence $r(\gamma)=w$. Multiplication on the left is similar. Hence the closed span is an ideal and so
\[ J_w=\clspan\set{S_\alpha S_\beta^*\colon \alpha,\beta\in E_2^*,\, r(\alpha)=r(\beta)=w}. \]
We have a system of matrix units since $(S_\alpha S_\beta^*)^*=S_\beta S_\alpha^*$ and for $r(\gamma)=r(\delta)=w$,
\[ S_\alpha S_\beta^* S_\gamma S_\delta^* = \casesotherd{S_\alpha S_\delta^*}{\beta=\gamma}{0} \]
Note that $\beta$ cannot extend $\gamma$ nor vice versa since $w$ is a sink. Hence $J_w\cong\K$.
\end{proof}

We define a representation $\pi\colon C^*(E_2)\to\B(\Hi)$ where $\Hi$ is a Hilbert space with orthonormal basis indexed by the set of finite paths in $E_2$ ending at $w$:
\[ \Hi=\clspan\set{\zeta_\alpha\colon \alpha\in E_2^*,\, r(\alpha)=w}. \]
Define
\[ \pi(S_f)\zeta_\alpha := \casesother{\zeta_{f\alpha}}{r(f)=s(\alpha)}{0} \]
for $f\in E_2^1$. Intuitively, $S_f$ adds the edge $f$, if possible. This definition satisfies the Cuntz-Krieger relations for $E_2$, and extends to a representation on $C^*(E_2)$ by the universal property. Similar representations have been used in the literature.

On inspection, all finite paths ending at $w$ have the form $w$, $b^ne$, $a^nd$ or $a^mcb^ne$ for $m,n\in\N$, where $\N$ is taken to include $0$. For example, the path $e$ is written $b^0e$, and $cb^3e$ is written $a^0cb^3e$. We wish to compute
\[ \pi(z_1) = \pi\left(\sum_{k=0}^\infty \lambda_k S_1^{k+1}(S_1^*)^k\right) = \sum_{k=0}^\infty \lambda_k \pi(S_1)^{k+1}\pi(S_1)^{*k}, \]
and $\pi(z_2)$, which has a similar form. Now
\begin{align*}
\pi(S_1)\colon & \zmapsto{w}{b^0e} \\
               & \zmapsto{b^ne}{b^{n+1}e},
\end{align*}
\begin{align*}
\pi(S_2)\colon & \zmapsto{w}{a^0d} \\
               & \zmapsto{b^ne}{a^0cb^ne} \\
               & \zmapsto{a^nd}{a^{n+1}d} \\
               & \zmapsto{a^mcb^ne}{a^{m+1}cb^ne}.
\end{align*}
Any basis elements omitted from the declaration of operators are presumed mapped to zero. For example, $\pi(S_2)$ sends no basis elements to zero, but $\pi(S_1)$ maps $\zeta_{a^nd}$ and $\zeta_{a^mcb^ne}$ to zero for all $m,n$.

It is easy to calculate adjoints in this representation. Each operator $Q$ in the representation that we will encounter has the property that it maps a basis vector to the scalar multiple of another basis vector, i.e.
\[ Q\zeta_\alpha=\lambda_\alpha\zeta_\beta, \]
where $\lambda_\alpha\in\R$. Also, at most one $\zeta_\alpha$ maps to any given $\zeta_\beta$, up to scalar multiples. This simplifies the computation of adjoints:
\begin{itemize}
\item $Q^*\zeta_\alpha=0$ if $\zeta_\alpha$ is not the image of any $\zeta_\beta$, up to scalar multiples. This follows from the fact that for an arbitrary operator $T$ on Hilbert space, $T^*$ sends $\im(T)^\perp$ to zero.
\item if $Q\zeta_\alpha=\lambda_\alpha\zeta_\beta$ then $Q^*\zeta_\beta=\lambda_\alpha\zeta_\alpha$, which is easily shown.
\end{itemize}
Hence
\[ \pi(S_1)^*\colon \zeta_{b^ne} \mapsto \zcasesif{w}{n=0}{b^{n-1}e}{n\ge 1,} \]
\begin{align*}
\pi(S_2)^*\colon & \zeta_{a^nd} \mapsto \zcasesif{w}{n=0}{a^{n-1}d}{n\ge 1,} \\
                 & \zeta_{a^mcb^ne} \mapsto \zcasesif{b^ne}{m=0}{a^{m-1}cb^ne}{m\ge 1.}
\end{align*}
All other basis elements are mapped to zero. Then
\begin{align*}
\pi(S_2)^{k+1}\colon & \zmapsto{w}{a^kd} \\
                     & \zmapsto{b^ne}{a^kcb^ne} \\
                     & \zmapsto{a^nd}{a^{n+k+1}d} \\
                     & \zmapsto{a^mcb^ne}{a^{m+k+1}cb^ne},
\end{align*}
and for $k\ge 1$,
\begin{align*}
\pi(S_2)^{*k}\colon & \zeta_{a^nd} \mapsto \zcasesif{w}{n=k-1}{a^{n-k}d}{n\ge k,} \\
                    & \zeta_{a^mcb^ne} \mapsto \zcasesif{b^ne}{m=k-1}{a^{m-k}cb^ne}{m\ge k.}
\end{align*}
The composition satisfies
\begin{align*}
\pi(S_2)^{k+1}\pi(S_2)^{*k}\colon & \zmapsto{a^nd}{a^{n+1}d} \qquad\text{if $n\ge k-1$} \\
                                  & \zmapsto{a^mcb^ne}{a^{m+1}cb^ne} \qquad\text{if $m\ge k-1$,}
\end{align*}
for $k\ge 1$. So any given basis vector $\zeta_\alpha$ gets mapped to zero if $k$ is large enough. Separating the $k=0$ term from the other terms in the expression for $\pi(z_2)$, we have
\[ \pi(z_2) = \sqrt{1-q}\,\pi(S_2) + \sum_{k=1}^\infty\lambda_k\pi(S_2)^{k+1}\pi(S_2)^{*k}. \]
The sum will be finite on any given basis element. In fact, 
\begin{align*}
\pi(z_2)\zeta_{a^nd} &= \sqrt{1-q}\pi(S_2)\zeta_{a^nd} + \sum_{k=1}^{n+1}\lambda_k\pi(S_2)^{k+1}\pi(S_2)^{*k}\zeta_{a^nd} \\
                     &= \sqrt{1-q}\zeta_{a^{n+1}d} + \sum_{k=1}^{n+1}\lambda_k\zeta_{a^{n+1}d} = \sqrt{1-q^{n+2}}\,\zeta_{a^{n+1}d}.
\end{align*}
Similarly $\pi(z_2)\zeta_{a^mcb^ne}=\sqrt{1-q^{m+2}}\,\zeta_{a^{m+1}cb^ne}$, and so we have
\begin{align*}
\pi(z_2)\colon & \zeta_w \mapsto \sqrt{1-q}\,\zeta_{a^0d} \\
               & \zeta_{b^ne} \mapsto \sqrt{1-q}\,\zeta_{a^0cb^ne} \\
               & \zeta_{a^nd} \mapsto \sqrt{1-q^{n+2}}\,\zeta_{a^{n+1}d} \\
               & \zeta_{a^mcb^ne} \mapsto \sqrt{1-q^{m+2}}\,\zeta_{a^{m+1}cb^ne}.
\end{align*}
after including the other two cases. By a similar process one can show
\begin{align*}
\pi(z_1)\colon & \zeta_w \mapsto \sqrt{1-q}\,\zeta_{b^0e} \\
               & \zeta_{b^ne} \mapsto \sqrt{1-q^{n+2}}\,\zeta_{b^{n+1}e},
\end{align*}
where as usual it is implied that the other basis elements are mapped to zero. These have adjoints
\[ \pi(z_1)^*\colon \zeta_{b^ne} \mapsto \casesif{\sqrt{1-q}\,\zeta_w}{n=0}{\sqrt{1-q^{n+1}}\,\zeta_{b^{n-1}e}}{n\ge 1,} \]
\begin{align*}
\pi(z_2)^* \colon & \zeta_{a^nd} \mapsto \casesif{\sqrt{1-q}\,\zeta_w}{n=0}{\sqrt{1-q^{n+1}}\,\zeta_{a^{n-1}d}}{n\ge 1,} \\
    & \zeta_{a^mcb^ne} \mapsto \casesif{\sqrt{1-q}\zeta_{b^ne}}{m=0}{\sqrt{1-q^{m+1}}\zeta_{a^{m-1}cb^ne}}{m\ge 1.}
\end{align*}
One may verify that $\pi(z_1z_2)=\pi(z_1)\pi(z_2)$ is the operator that sends all basis elements to zero, hence $z_1z_2=0$ since $\pi$ is faithful. Similarly $z_1^*z_2=0$, and so we have the relations
\begin{align}
z_1z_2     &= 0              \label{B4eq1} \\
z_1^*z_2   &= 0              \label{B4eq2} \\
\qrel{z_2} &= (1-q)1         \label{B4eq3} \\
\qrel{z_1} &= (1-q)(P_v+P_w) \label{B4eq4}
\end{align}
where the last two follow from the representations of the following elements:
\begin{align*}
\pi(z_1^*z_1) \colon & \zeta_w \mapsto (1-q)\zeta_w \\
                     & \zeta_{b^ne} \mapsto (1-q^{n+2})\zeta_{b^ne}.
\end{align*}
\[ \pi(z_1z_1^*) \colon \zeta_{b^ne} \mapsto (1-q^{n+1})\zeta_{b^ne}. \]
\begin{align*}
\pi(z_2^*z_2) \colon & \zeta_w \mapsto (1-q)\zeta_w \\
                     & \zeta_{b^ne} \mapsto (1-q)\zeta_{b^ne} \\
                     & \zeta_{a^nd} \mapsto (1-q^{n+2})\zeta_{a^nd} \\
                     & \zeta_{a^mcb^ne} \mapsto (1-q^{m+2})\zeta_{a^mcb^ne}.
\end{align*}
\begin{align*}
\pi(z_2z_2^*) \colon & \zeta_{a^nd} \mapsto (1-q^{n+1})\zeta_{a^nd} \\
                     & \zeta_{a^mcb^ne} \mapsto (1-q^{m+1})\zeta_{a^mcb^ne}.
\end{align*}
Note that Equation \eqref{B4eq3} is the quantum relation and Equation \eqref{B4eq4} is the quantum relation on a subspace.
\begin{lemma} \label{lemB4b}
$C^*(E_2)=C^*(z_1,z_2)$; in other words, $C^*(E_2)$ is generated by $z_1$ and $z_2$.
\end{lemma}

\begin{proof}
We show that $C^*(S_1,S_2)=C^*(z_1,z_2)$, from which the result follows. We must show that $S_1$ and $S_2$ can be recovered from $z_1$ and $z_2$. Notice that $z_2^*z_2$ is invertible, from its image in the representation. Hence it has a unique polar decomposition with phase $\phase{z_2}$, which one may show is just $S_2$. In particular, $S_2\in C^*(z_1,z_2)$.

Now $z_1^*z_1$ is not invertible, however it is invertible in the \cstsubalg\ $(P_v+P_w)C^*(E_2)(P_v+P_w)$ of $C^*(E_2)$ which is unital with unit $P_v+P_w$. Working in this \cstsubalg, $z_1$ has a unique polar decomposition with phase $\phase{z_1}=\phaseE{z_1}$ which, in the representation restricted to the subalgebra, is the operator
\[ \pi(z_1)\pi(z_1^*z_1)^{-1/2} \colon \bigbrace{\zeta_w}{\zeta_{b^ne}} \xmapsto{\pi(z_1^*z_1)^{-1/2}} \bigbrace{(1-q)^{-1/2}\zeta_w}{(1-q^{n+2})^{-1/2}\zeta_{b^ne}} \xmapsto{\pi(z_1)} \bigbrace{\zeta_{b^0e}}{\zeta_{b^{n+1}e},} \]
which is just $\pi(S_1)$ since $S_1$ also belongs to the subalgebra. Hence $S_1$ is the phase of $z_1$. It remains to show that $P_v+P_w\in C^*(z_1,z_2)$, which follows since $P_v+P_w=1-P_u=1-S_2S_2^*\in C^*(z_1,z_2)$.
\end{proof}

\begin{proof}[Proof of Theorem \ref{thrmE2}]
We want to find homomorphisms
\[ \morphgraph{C^*(E_2)}{\qquad C^*(x_1,x_2)} \]
so we must find images for the generators. It is logical to define
\[ \psi(x_i) := z_i, \qquad i=1,2, \]
as $x_i$ and $z_i$ satisfy the same relations. To define $\phi$ we must find images for the Cuntz-Krieger generators. Note that we must go back to these generators since we have not defined $C^*(z_1,z_2)$ as a universal algebra with respect to certain relations between $z_1$ and $z_2$. It is the original generators and relations which define the universal algebra $C^*(E_2)$.

Lemma \ref{lemB4a} gives the generators in terms of $S_i$, $i=1,2$. Now since $S_i$ is the phase of $z_i$, $i=1,2$, it makes sense that $\phi$ should map each $S_i$ to the phase of $x_i$, if these admit a polar decomposition. In fact $x_2$ has a polar decomposition since $\abs{x_2}^2 = x_2^*x_2 = qx_2x_2^*+(1-q)1 > (1-q)1$, so $\abs{x_2}$ is invertible. Let $\alpha_2$ be the phase of $x_2$, so $\alpha_2=\phase{x_2}$ and $\alpha_2^*\alpha_2=1$. Notice that $\alpha_2\alpha_2^* = x_2\abs{x_2}^{-2}x_2^* = 1-P$. Also $x_1$ belongs to the \cstsubalg\ $PC^*(x_1,x_2)P$, and $\abs{x_1}$ is invertible in this subalgebra by the same reasoning. Hence $x_1$ has a polar decomposition $x_1=\alpha_1\abs{x_1}$ say, where the phase $\alpha_1$ satisfies $\alpha_1^*\alpha_1=P$.

We use the results of Lemma \ref{lemB4a} to define $\phi$:
\[ \phi(P_u):=\alpha_2\alpha_2^*, \qquad \phi(P_v):=\alpha_1\alpha_1^* \quad \cdots \quad \phi(S_e):=\alpha_1(\alpha_1^*\alpha_1-\alpha_1\alpha_1^*). \]
The images are clearly projections and partial isometries respectively. To show that $\phi$ and $\psi$ extend to homomorphisms, we must show that they preserve the relations. Equations \eqref{B4eq1}-\eqref{B4eq4} show that $\psi$ preserves the relations for $C^*(x_1,x_2)$, if we can show that $P_v+P_w=\psi(1-\projrel{x_2})$. But the right hand side is just
\[ 1-\projrel{z_2} = 1-\phase{z_2}\abs{z_2}^{-1}z_2^* = 1-S_2S_2^* = 1-P_u = P_v+P_w, \]
by Lemmas \ref{lemB4a} and \ref{lemB4b}. We can show that $\phi$ preserves the Cuntz-Krieger relations. Define $Q:=\alpha_1\alpha_1^*$. Then $Q\le P$, and in fact $Q<P$ as we will see shortly. We then have
\[ \phi(P_u) = 1-P, \qquad \phi(P_v) = Q, \qquad \phi(P_w) = P-Q. \]
We illustrate that $\phi$ preserves the Cuntz-Krieger relations with some examples. Firstly $P_vP_w=0$, and
\[ \phi(P_v)\phi(P_w) = Q(P-Q) = QP-Q^2 = Q-Q = 0, \]
as required. Also $S_d^*S_d=P_w$, and
\begin{align*}
\phi(S_d)^*&\phi(S_d) = (\alpha_2(\alpha_1^*\alpha_1-\alpha_1\alpha_1^*))^*\alpha_2(\alpha_1^*\alpha_1-\alpha_1\alpha_1^*) = (\alpha_2(P-Q))^*\alpha_2(P-Q) \\
&= (P-Q)\alpha_2^*\alpha_2(P-Q) = (P-Q)^2 = P-Q = \phi(P_w).
\end{align*}
Also $P_v=S_bS_b^*+S_eS_e^*$, and we have
\begin{align*}
\phi(S_b)&\phi(S_b)^*+\phi(S_e)\phi(S_e)^* = \alpha_1^2\alpha_1^*\alpha_1(\alpha_1^*)^2 + \alpha_1(1-\alpha_1\alpha_1^*)(1-\alpha_1\alpha_1^*)\alpha_1^* \\
&= \alpha_1^2P(\alpha_1^*)^2 + \alpha_1(1-Q)^2\alpha_1^* = \alpha_1Q\alpha_1^*+\alpha_1(1-Q)\alpha_1^* = \alpha_1\alpha_1^* = Q = \phi(P_v),
\end{align*}
where we have used $\alpha_1P\alpha_1^*=\alpha_1\alpha_1^*=Q$. The other relations may be shown in a similar manner.

To show that $\psi\circ\phi=\id_{C^*(E_2)}$ we first show that $\phi(S_i)=\alpha_i$ for $i=1,2$, which holds since
\[ \phi(S_1) = \phi(S_b+S_e) = \alpha_1(\alpha_1\alpha_1^*+1-\alpha_1\alpha_1^*) = \alpha_1, \]
and
\[ \phi(S_2) = \phi(S_a+S_c+S_d) = \alpha_2(\alpha_2\alpha_2^*+\alpha_1\alpha_1^*+\alpha_1^*\alpha_1-\alpha_1\alpha_1^*) = \alpha_2(1-P+P) = \alpha_2. \]
Hence
\[ (\psi\circ\phi)(S_i) = \psi(\alpha_i) = \psi(\Ph(x_i)) = \Ph(\psi(x_i)) = \Ph(z_i) = S_i, \]
for $i=1,2$, where $\Ph$ denotes phase. Consequently $\psi\circ\phi=\id$ on all the Cuntz-Krieger generators, since these are determined by $S_1$ and $S_2$, and so $\psi\circ\phi=\id_{C^*(E_2)}$.

To show that $\phi\circ\psi=\id_{C^*(x_1,x_2)}$, we first need to show $C^*(x_i)\cong C(\D_q)$, $i=1,2$. For $z_2$, we know $C(\D_q)\cong C^*(S_2)\cong C^*(z_2)$, so the map $x\mapsto z_2$ is an isomorphism, and $z_2$ is like the quantum generator. By universality of the quantum disc there exists a surjective homomorphism $\eta\colon C^*(z_2)\to C^*(x_2)$ which sends $z_2\mapsto x_2$. Also, $\psi$ can be restricted to a homomorphism $\psi\colon C^*(x_2)\to C^*(z_2)$, $x_2\mapsto z_2$. The images of the generators are nonzero, the quantum relation is preserved, and the compositions give the identity maps. Hence $\psi$ and $\eta$ are isomorphisms between these subalgebras, so $C^*(x_2)\cong C(\D_q)$\footnote{Alternatively, $\phi$ is an injective homomorphism, which we can restrict to a map $\phi\colon C^*(z_i)=C^*(S_i)\to C^*(\alpha_i)\subset C^*(x_i)$. By injectivity of $\phi$, $C^*(x_i)$ contains a copy of $C(\D_q)$. It can be no larger than $C(\D_q)$, since there is a surjective homomorphism from $C(\D_q)$ onto $C^*(x_i)$ by universality. Hence we must have isomorphism: $C^*(x_i)\cong C(\D_q)$.}. We can show the same result for $C^*(x_1)$. Firstly $P\in C^*(x_1)$ since $P=(1-q)^{-1}(x_1^*x_1-qx_1x_1^*)$. Then $C^*(x_1)$ is contained in the \cstsubalg\ $PC^*(x_1,x_2)P$, and one can work in this subalgebra to obtain the desired result.

One consequence of the isomorphism $C^*(x_i)\cong C(\D_q)$ is that $\alpha_i\alpha_i^*<\alpha_i^*\alpha_i$, and consequently the images under $\phi$ of the Cuntz-Krieger generators are nonzero. We can now show $\phi\circ\psi=\id_{C^*(x_1,x_2)}$ since
\[ (\phi\circ\psi)(x_i) = \phi(z_i) = \phi\left(\sum_{k=0}^\infty \lambda_k S_i^{k+1}S_i^{*k}\right) = \sum_{k=0}^\infty \lambda_k \alpha_i^{k+1}\alpha_i^{*k}. \]
By the isomorphism just exhibited, we can compute in $C^*(z_i)\subset C^*(E_2)$ to show that this sum is just $x_i$. So $\phi\circ\psi=\id_{C^*(x_1,x_2)}$ because it is the identity on generators. Consequently $\phi$ and $\psi$ are injective and surjective, and hence $C^*(E_2)\cong C^*(x_1,x_2)$.
\end{proof}

\subsection{Representations}

We classify the irreducible representations of $C^*(x_1,x_2)\cong C(B_q^4)$ up to unitary equivalence, using the general theory of graph algebras.

\begin{thrm}
Every irreducible \strepn\ of $C(\B_q^4)$ is unitarily equivalent to one of the following representations:
\begin{enumerate}
\item a representation $\pi$ defined on a Hilbert space $\Hi$ with orthonormal basis $\set{\zeta_i}$ indexed by paths of finite length in $E_2$ ending at $w$: $w$, $b^ne$, $a^nd$, or $a^mcb^ne$ where $m,n\in\N=\N\cup\set{0}$ by:
\begin{align*}
\pi(x_1)\colon & \zeta_w \mapsto \sqrt{1-q}\,\zeta_{b^0e} \\
               & \zeta_{b^ne} \mapsto \sqrt{1-q^{n+2}}\,\zeta_{b^{n+1}e} \\
               & \zeta_{a^nd}, \zeta_{a^mcb^ne} \mapsto 0, \\
\pi(x_2)\colon & \zeta_w \mapsto \sqrt{1-q}\,\zeta_{a^0d} \\
               & \zeta_{b^ne} \mapsto \sqrt{1-q}\,\zeta_{a^0cb^ne} \\
               & \zeta_{a^nd} \mapsto \sqrt{1-q^{n+2}}\,\zeta_{a^{n+1}d} \\
               & \zeta_{a^mcb^ne} \mapsto \sqrt{1-q^{m+2}}\,\zeta_{a^{m+1}cb^ne},
\end{align*}
\item a family of representations $\varepsilon_\theta$ parameterised by $\theta\in S^1$ on the Hilbert space $\Hi$ with orthonormal basis $\set{\zeta_v}\cup\set{\zeta_{a^nc}\colon n\in\N}$: 
\begin{align*}
\varepsilon_\theta(x_1) \colon & \zeta_v \mapsto \theta\zeta_v \\
                               & \zeta_{a^nc} \mapsto 0, \\
\varepsilon_\theta(x_2) \colon & \zeta_v \mapsto \sqrt{1-q}\zeta_{a^0c} \\
                               & \zeta_{a^nc} \mapsto \sqrt{1-q^{n+2}}\zeta_{a^{n+1}c},
\end{align*}
\item a family of one-dimensional representations $\sigma_\theta$ parameterised by $\theta\in S^1$: 
\[ \sigma_\theta(x_1):=0, \qquad \sigma_\theta(x_2):=\theta. \]
\end{enumerate}
Only the representation $\pi$ is faithful.
\end{thrm}

\begin{proof}
$C^*(x_1,x_2)\cong C^*(E_2) = C^*\Bigl(\graphEfourv\Bigr)$.\vspace{-3mm} Hence the ideals and representations of $C^*(x_1,x_2)$ correspond exactly to the ideals and representations of $C^*(E_2)$, which we can obtain using the powerful general theory of graph algebras. As already shown, $\set{w}$ is hereditary and saturated, and the ideal $J_w$ generated by $P_w$ is isomorphic to the compacts. The quotient $C^*(E_2)/J_w$ is isomorphic to the \cstalg\ $C^*(E_2\backslash\set{w})$, so we have an exact sequence
\[ 0\to J_w\to C^*(E_2)\to C^*\Bigl(\graphSthree\Bigr)\to 0. \vspace{-8mm}\]

Hence $\Irr(C^*(E_2))=\Irr(J_w)\sqcup\Irr(\graphSthreeS)$,\vspace{-8mm} where $\Irr$ denotes the irreducible representations. There is only one irreducible representation of $J_w$ up to unitary equivalence, which when extended to the whole algebra is just the representation $\pi$ from before. Composing with the isomorphism $\psi$ gives the representation stated in the Theorem. 

Now $C^*\Bigl(\graphSthree\Bigr)$\vspace{-8mm} is the quantum sphere $C(S_q^3)$, which is the most well understood quantum object. So the exact sequence becomes
\[ 0\to\K\to C(B_q^4)\to C(S_q^3)\to 0, \]
which is a quantum analogue of the classical case. The classical picture is $\R^4\to B^4\to S^3$, in other words the interior $\Int(B^4)$ is isomorphic to $\R^4$ and the boundary $\delta(B^4)$ is isomorphic to $S^3$. Here, the Weyl quantisation of $\R^4$ is $\K$, the quantum ball $C(B_q^4)$ is the quantum deformation of the ball $B^4$, and the quantum sphere $C(S_q^3)$ is the quantum analogue of the sphere $S^3$. So we have a nice analogy with the classical case.

The quotient \cstalg\ $C(S_q^3)$ is not simple, since $\set{v}$ is hereditary and saturated. Let $J_v$ be the ideal generated by $P_v$ in this algebra. We have an exact sequence
\[ 0\to J_v\to C^*\Bigl(\graphSthree\Bigr)\to C^*\Bigl(\graphloop\Bigr)\to 0.\vspace{-10mm} \]
Hence $\Irr(C(S_q^3))=\Irr(J_v)\sqcup\Irr(\graphloopS)$\vspace{-9mm}. The ideal $J_v$ is isomorphic to $K\otimes C(S^1)$ by the following reasoning. Consider all paths $\alpha,\beta$ in the graph for $C(S_q^3)$ with $r(\alpha)=r(\beta)=v$ and such that neither path contains $b$. Then
\[ \set{S_\alpha S_\beta^*\colon r(\alpha)=r(\beta)=v;\; \alpha,\beta\text{ do not contain }b} \]
is a system of matrix units in $J_v$. Hence their closed span is isomorphic to $\K$. Notice that all such paths have the form $a^nc$ for $n\in\N$ or consist of just the single vertex $\set{v}$.

Also $S_b\in J_v$ since $S_b=S_bP_v$. $S_b$ is a partial unitary with $S_bS_b^*=S_b^*S_b=P_v$, and full spectrum $\sigma(S_b)=S^1\cup\set{0}$. Hence $C^*(S_b)\cong C(S^1)$. Now $S_b$ commutes with $S_\alpha S_\beta^*$, so by the general theory of ideals in graph algebras $J_v\cong\K\otimes C(S^1)$.

This ideal has $S^1$ of irreducible representations (in other words, the irreducible representations of this ideal can be parameterised by $S^1$). Fix $\theta\in S^1\subset\C$. Define a representation $\varepsilon_\theta\colon J_v\to\B(\Hi)$, where $\Hi$ is the Hilbert space with orthonormal basis indexed by paths $a^nc$ or $v$, and extend it to a representation on the whole algebra $C(S_q^3)$:

\begin{align*}
\varepsilon_\theta(S_a) \colon & \zeta_v \mapsto 0 \\
                               & \zmapsto{a^nc}{a^{n+1}c}, \\
\varepsilon_\theta(S_c) \colon & \zmapsto{v}{a^0c} \\
                               & \zeta_{a^nc} \mapsto 0, \\
\varepsilon_\theta(S_b) \colon & \zeta_v \mapsto \theta\zeta_v \\
                               & \zeta_{a^nc} \mapsto 0.
\end{align*}
Now compose with the quotient map $\colon C^*(E_2)\to C(S_q^3)$ to get a representation on $C^*(E_2)$. In terms of $S_1$ and $S_2$,
\begin{align*}
\varepsilon_\theta(S_1) \colon & \zeta_v \mapsto \theta\zeta_v \\
                               & \zeta_{a^nc} \mapsto 0, \\
\varepsilon_\theta(S_2) \colon & \zmapsto{v}{a^0c} \\
                               & \zmapsto{a^nc}{a^{n+1}c}.
\end{align*}
Note that $\varepsilon_\theta(S_2)$ is almost identical to the representation $\pi(S_1)$ used previously, where only the labels are different. So by analogous computations and composing with the isomorphism $\psi$, we obtain $\varepsilon_\theta(x_2)$ as in the Theorem. For $\varepsilon_\theta(S_1)$ we have $\varepsilon_\theta(S_1)^*\zeta_v = \bar\theta\zeta_v$, and so $\varepsilon_\theta(S_1)^{*k}\zeta_v = \bar\theta^k\zeta_v$. Also $\varepsilon_\theta(S_1)^{k+1}\zeta_v = \theta^{k+1}\zeta_v$, and so
\[ \varepsilon_\theta(S_1)^{k+1}\varepsilon_\theta(S_1)^{*k}\zeta_v = \theta^{k+1}\bar\theta^k\zeta_v = \theta\zeta_v, \]
since $\theta\bar\theta=1$. Hence
\[ \varepsilon_\theta(z_1)\zeta_v = \sum_{k=0}^\infty \lambda_k\varepsilon_\theta(S_1)^{k+1}\varepsilon_\theta(S_1)^{*k}\zeta_v = \sum_{k=0}^\infty \lambda_k\theta\zeta_v = \left(\sum_{k=0}^\infty \lambda_k\right)\theta\zeta_v = \theta\zeta_v,  \]
because $\sum_{k=0}^\infty \lambda_k=1$. Composing with $\psi$ gives the stated representation.

Finally, we need to classify the representations of $C^*\Bigl(\graphloop\Bigr)$.\vspace{-9mm} This \cstalg\ is isomorphic to $C(S^1)$, which has $S^1$ of irreducible representations. Pick $\theta\in S^1$, and define a representation $\sigma_\theta\colon C(S^1)\to\C$ by
\[ \sigma_\theta(S_a):=\theta, \]
which are the characters of $C(S^1)$. We have $\sigma_\theta(P_u)=1$, and we can compose with the quotient maps to get a representation on $C^*(E_2)$ and then compose with $\psi$ to get the representation of $C^*(x_1,x_2)$ shown in the Theorem. These are all the irreducible representations of $C^*(x_1,x_2)$.
\end{proof}

The representations $\pi$, $\varepsilon_\theta$ and $\sigma_\theta$ comprise all the irreducible representations of $C^*(x_1,x_2)$. Effectively we have a point and two circles of representations:
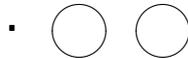
\[ \centerdot \qquad \text{\circle{20} } \qquad \text{ \circle{20}} \]
We now proceed to the general case $C(B_q^{2n})$.

\section{Quantum \texorpdfstring{$2n$}{2n}-ball \texorpdfstring{$C(B_q^{2n})$}{Cbq2n}}

Let $E_n$ denote the graph obtained by applying the double suspension $n$ times to the single vertex graph. Then $C(B_q^{2n})=C^*(E_n)$ by definition. $E_n$ has $n+1$ vertices which we will label $v_0$ to $v_n$. There is an edge from $v_i$ to $v_j$ labelled $e_{ij}$ if and only if $i\ge j$ and $i\ne 0$:
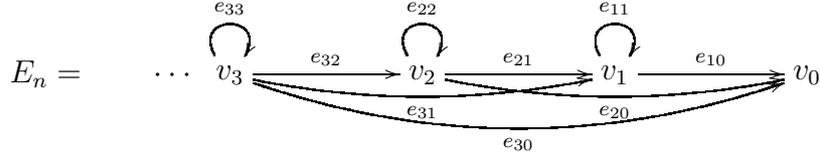
\[ E_n= \qquad \dotsb
\xygraph{{v_3}="v_3"(:^{e_{33}}@(ul,ur)"v_3" :^{e_{32}}[r(2)]{v_2}="v_2"( :^{e_{22}}@(ul,ur)"v_2" :^{e_{21}}[r(2)]{v_1}="v_1" :^{e_{11}}@(ul,ur)"v_1" :^{e_{10}}[r(2)]{v_0}="v_0", :_{e_{20}}@/_0.7pc/"v_0"), :_{e_{31}}@/_0.7pc/"v_1", :_{e_{30}}@/_1.7pc/"v_0")}
\]
The graph algebra $C^*(E_n)$ is generated by projections $\set{P_i\colon i=0,\dotsc,n}$ and partial isometries $\set{S_{ij}\colon 0\le j\le i\le n}$ subject to the relations
\[ P_iP_j=0 \,\text{ for } i\ne j, \qquad S_{ij}^*S_{ij}=P_j, \qquad P_i=\sum_{j=0}^i S_{ij}S_{ij}^* \,\text{ for }i\ne 0. \]

\subsection{Polynomial Algebra}

\begin{defn}
Consider the free \stalg\ $\F(x_1,\dotsc,x_n)$ and all representations $\pi^{(n)}$ of this algebra into bounded operators on a Hilbert space which satisfy the relations
\begin{align}
x_ix_j &= 0 \,\text{ for } i<j \\
x_i^*x_j &= 0 \,\text{ for } i\ne j \\
\label{eqn:quantumrelations}\qrel{x_n} &= (1-q)1.
\end{align}
In all such representations, the element $\pi^{(n)}(x_n^*x_n)$ is invertible. Hence $x_n$ has a polar decomposition $x_n=\alpha_n\abs{x_n}$, where $\alpha_n=\phase{x_n}$. Define a projection $Q_{n-1}:=1-\alpha_n\alpha_n$. This is nonzero since $\alpha_n\alpha_n^*<1$:
\[ \alpha_n\alpha_n^*=1 \imp x_n\abs{x_n}^{-2}x_n^*=1 \imp x_{n-1}x_n\abs{x_n}^{-2}x_n^* = x_{n-1}, \]
which is a contradiction since $x_{n-1}x_n=0$, and all $x_i$ are nonzero because they are nonzero in the representation $\pi$. Hence we can consider the smaller set of representations $\pi^{(n-1)}$ which satisfy the additional relation
\[ \qrel{x_{n-1}} = (1-q)Q_{n-1}. \]
Then $x_{n-1}Q_{n-1}=x_{n-1}(1-x_n\abs{x_n}^{-2}x_n^*)=x_{n-1}$ and similarly $Q_{n-1}x_{n-1}=x_{n-1}$, so $x_{n-1}$ belongs to the \cstsubalg\ $Q_{n-1}\Hi Q_{n-1}$ in all representations. In this subalgebra the above relation becomes $\qrel{x_{n-1}}=(1-q)1$, and so $x_{n-1}$ has a polar decomposition $x_{n-1}=\alpha_{n-1}\abs{x_{n-1}}$. The phase satisfies $\alpha_{n-1}=x_{n-1}\widetilde{x_{n-1}}$, where $\widetilde{x_{n-1}}$ is the inverse of $\abs{x_{n-1}}$ in this subalgebra. Also $\alpha_{n-1}^*\alpha_{n-1}=Q_{n-1}$, and $\alpha_{n-1}\alpha_{n-1}^*<Q_{n-1}$ since otherwise
\[ \alpha_{n-1}\alpha_{n-1}^*=Q_{n-1} \imp Q_{n-2}=0 \imp \alpha_{n-2}^*\alpha_{n-2}=0 \imp \alpha_{n-2}=0 \imp x_{n-2}=0. \]
So we define the nonzero projection $Q_{n-2}:=Q_{n-1}-\alpha_{n-1}\alpha_{n-1}^*$, and consider the smaller set of representations $\pi^{(n-2)}$ which also satisfy the relation
\[ \qrel{x_{n-2}} = (1-q)Q_{n-2}, \]
and continue this procedure until we have the set of representations $\pi^{(1)}$ or $\pi$, satisfying all the relations
\begin{equation}
\qrel{x_i} = (1-q)Q_i, \qquad \text{for } i=1,\dotsc,n.
\end{equation}
We define $Q_n:=1$ for consistency. We have
\[ Q_{i-1} = Q_i-\alpha_i\alpha_i^* = \alpha_i^*\alpha_i-\alpha_i\alpha_i^* \qquad\text{for } i=2,\dotsc,n. \]
In every representation $\pi$ we have $\norm{\pi(x_i)}=1$ for all $i$. Hence we can define a semi-norm $\norm{\cdot}_0$ on $\F(x_1,\dotsc,x_n)$ by
\[ \norm{a}_0 := \sup\set{\norm{\pi(a)}}. \]
Let $J_0 := \set{a\in\F(x_1,\dotsc,x_n)\colon \norm{a}_0=0}$. One may check this is a two-sided \stideal\ of $\F(x_1,\dotsc,x_n)$. Define the \emph{polynomial algebra} $P(B_q^{2n})$ of the quantum $2n$-ball to be the quotient \stalg\ $\F(x_1,\dotsc,x_n)/J_0$.
\end{defn}

\begin{thrm} \label{thrmEn}
The completion $C^*(x_1,\dotsc,x_n)$ of the polynomial algebra $P(B_q^{2n})$ is isomorphic to the quantum $2n$-ball $C(B_q^{2n})$.
\end{thrm}

\begin{lemma}
Define $R_i := \alpha_i\alpha_i^*$. Then we have the following relations
\begin{enumerate}
\item $Q_{i-1}=Q_i-R_i$
\item $Q_{i-1}=1-\sum_{j=i}^n R_j$ for $i=2,\dotsc,n$
\item $R_i<Q_j$ for $i\le j$
\item $R_iR_j=0$ for $i\ne j$
\item $Q_iR_j=0$ for $i<j$
\end{enumerate}
\end{lemma}

\begin{proof}
\begin{enumerate}
\item Follows from the definition.
\item $Q_{i-1} = Q_i-R_i = Q_{i+1}-R_{i+1}-R_i = Q_{i+2}-R_{i+2}-R_{i+1}-R_i = \dotsb = Q_n-R_n-\dotsb-R_i = 1-\sum_{j=i}^n R_j$.
\item $R_i<Q_i$ already shown for $i>1$. Also we will see in the proof of the Theorem that $R_1<Q_1$, since $C^*(x_1)\cong C(\D_q)$. Then $R_i<Q_i\le Q_j$ for $i\le j$.
\item Suppose $1<i<j$. Then $R_iR_j = (Q_i-Q_{i-1})(Q_j-Q_{j-1}) = Q_i-Q_i-Q_i+Q_i = 0$. If $1=i<j$ then $R_1R_j = R_1(Q_j-Q_{j-1}) = R_1-R_1 = 0$.
\item $Q_iR_j = (1-\sum_{k=i+1}^n R_k)R_j = R_j-R_j = 0$.
\end{enumerate}
\end{proof}

Notice that the hereditary saturated subsets of $E_n^0$ are $\set{v_0,\dotsc,v_i}$ for $i=0,\dotsc,n$. Hence we define $S_1,\dotsc,S_n$ by
\[ S_i := \sum_{j=0}^i S_{ij} = S_{i0}+S_{i1}+\dotsb+S_{ii}, \]
similarly to before.

\begin{lemma}\label{lemBna}
$C^*(E_n)=C^*(S_1,\dotsc,S_n)$.
\end{lemma}

\begin{proof}
By the Cuntz-Krieger relations,
\[ S_i^*S_i = S_{i0}^*S_{i0}+\dotsb+S_{ii}^*S_{ii} = P_0+\dotsb+P_i. \]
Also
\[ S_iS_i^* = \sum_{j,k=0}^i S_{ij}S_{ik}^* = \sum_{j=0}^i S_{ij}S_{ij}^* = P_i \]
for $i=1,\dotsc,n$ and we also have $P_0 = (P_0+P_1)-P_1 = S_1^*S_1-S_1S_1^*$. So we can recover the projections. Then
\[ S_{ij} = S_{ij}P_j = S_iP_j, \]
so we can recover the partial isometries as well. 
\end{proof}

The above results show that $S_i$ is a partial isometry for $i=1,\dotsc,n-1$ and that $S_n$ is a proper isometry. Now define $z_1,\dotsc,z_n$ by
\[ z_i := \sum_{k=0}^\infty \lambda_k S_i^{k+1}(S_i^*)^k. \]
We again make computations in a representation $\pi\colon C^*(E_n)\to\B(\Hi)$, where $\Hi$ is the Hilbert space with orthonormal basis indexed by the paths of finite length in $E_n$ ending at $v_0$. We define $\pi$ on $\set{S_f\colon f\in E_n^1}$ by
\[ \pi(S_f)\zeta_\mu := \casesother{\zeta_{f\mu}}{r(f)=s(\mu)}{0} \]
which extends to a representation on all of $C^*(E_n)$ by the universal property. It can be shown that $\pi$ satisfies the Cuntz-Krieger relations. This representation is both faithful and irreducible. Intuitively, $S_f$ adds the edge $f$, if possible.

\subsubsection{Paths in $E_n$}
We can list all the finite paths in $E_n$ ending at $v_0$ by a recursive procedure. Let $\Gamma_i$ be the set of such paths starting at $v_i$:
\[ \Gamma_i := \set{\mu\in E_n^*\colon s(\mu)=v_i,\, r(\mu)=v_0}. \]
Then $\Gamma_n := \bigcup_{i=0}^n \Gamma_i$ is the set of all paths ending at $v_0$. There is only one such path starting at $v_0$, the $0$-length path consisting of the vertex only, so $\Gamma_0=\set{v_0}$. Paths starting at $v_1$ have the form $e_{11}^{m_1}e_{10}\in\Gamma_1$ for some $m_1\in\N$, where $\N$ includes $0$. To list the set of all paths starting at $v_2$, note that such paths have the form $e_{22}^{m_2}e_{20}$ or $e_{22}^{m_2}e_{10}\Gamma_1$, where the operation is concatenation, i.e. follow $e_{22}^{m_2}e_{10}$ by any path starting at $v_1$. In general,
\[ \Gamma_i = \bigcup_{j=0}^{i-1} e_{ii}^{m_i}e_{ij}\Gamma_j, \]
where $m_i\in\N$.

We can categorise the paths as follows. Say two paths are equivalent if they differ only by the number of loops around a vertex, that is, once the loops are removed the paths are identical. The number of equivalence classes of paths starting from $v_i$ for $i\ge 1$ is $2^{i-1}$. Hence the total number of classes of paths is $2^n$.

\subsubsection{Representation $\pi$ of $C^*(E)$}

We know $\pi(S_{ij})$ acts by adding the edge $e_{ij}$, if possible. So $\pi(S_{ij})$ is nonzero precisely on paths which start at $v_j$. Such paths have the form $e_{jj}^{m_j}\cdots$ and
\[ \pi(S_{ij})\zeta_{e_{jj}^{m_j}\cdots} = \casesif{\zeta_{e_{ii}^{m_i+1}\cdots}}{i=j}{\zeta_{e_{ii}^0e_{ij}e_{jj}^{m_j}\cdots}}{i\ne j.} \]
Now $\pi(S_i)$ acts by adding an edge from the set $\set{e_{ij}\colon j\le i}$. There is at most one such edge from this set which can be added to a path. In fact $\pi(S_i)$ is nonzero precisely on paths starting at $v_j$ where $j\le i$, in which case
\[ \pi(S_i)\zeta_{e_{jj}^{m_j}\cdots} = \casesif{\zeta_{e_{ii}^{m_i+1}\cdots}}{i=j}{\zeta_{e_{ii}^0e_{ij}e_{jj}^{m_j}\cdots}}{i>j.} \]
So
\[ \pi(S_i)^{k+1}\zeta_{e_{jj}^{m_j}\cdots} = \casesif{\zeta_{e_{ii}^{m_i+k+1}\cdots}}{i=j}{\zeta_{e_{ii}^ke_{ij}e_{jj}^{m_j}\cdots}}{i>j.} \]
$\pi(S_i)^*$ acts by removing an edge starting at $v_i$, if possible. Hence it is nonzero precisely on paths $e_{ii}^{m_i}e_{ij}\mu$ starting at $v_i$, where $\mu$ is the tail of the path, and then
\[ \pi(S_i)^*\zeta_{e_{ii}^{m_i}e_{ij}\mu} = \casesif{\zeta_\mu}{m_i=0}{\zeta_{e_{ii}^{m_i-1}e_{ij}\mu}}{m_i\ge 1,} \]
so
\[ \pi(S_i)^{*k}\zeta_{e_{ii}^{m_i}e_{ij}\mu} = \casesif{\zeta_\mu}{m_i=k-1}{\zeta_{e_{ii}^{m_i-k}e_{ij}\mu}}{m_i\ge k,} \]
and is zero if $m_i<k-1$ and on all other basis elements. So the composition satisfies
\[ \pi(S_i)^{k+1}\pi(S_i)^{*k}\zeta_{e_{ii}^{m_i}\cdots} = \zeta_{e_{ii}^{m_i+1}\cdots} \]
if $k\le m_i+1$, and is zero otherwise. Hence the sum expressing $\pi(z_i)$ will be finite on any given basis element $\zeta_\alpha$.

Separating the $k=0$ term from the other terms in the expression for $\pi(z_i)$, we have
\[ \pi(z_i) = \sqrt{1-q}\,\pi(S_i) + \sum_{k=1}^{m_i+1}\lambda_k\pi(S_i)^{k+1}\pi(S_i)^{*k}, \]
and so
\[ \pi(z_i)\zeta_{e_{jj}^{m_j}\cdots} = \begin{cases}
  \sqrt{1-q}\zeta_{e_{ii}^0e_{ij}e_{jj}^{m_j}\cdots} & \text{if } j<i \\
  \sqrt{1-q^{m_i+2}}\zeta_{e_{ii}^{m_i+1}\cdots}     & \text{if } j=i \\
  0                                                  & \text{if } j>i,
\end{cases} \]
by similar calculations to those done for $C(B_q^4)$. Hence the adjoint $\pi(z_i)^*$ is nonzero only on paths starting with $e_{ii}^{m_i}$, and satisfies
\[ \pi(z_i)^*\zeta_{e_{ii}^{m_i}\mu} = \casesif{\sqrt{1-q}\zeta_\mu}{m_i=0}{\sqrt{1-q^{m_i+1}}\zeta_{e_{ii}^{m_i-1}\mu}}{m_i\ge 1,} \]
on such paths, where $\mu$ is the tail of the path. One may then show the results
\begin{align}
x_ix_j     &= 0 \quad\text{for }i<j     \label{Bneq1} \\
x_i^*x_j   &= 0 \quad\text{for }i\ne j  \label{Bneq2} \\
\qrel{x_i} &= (1-q)(P_0+P_1+\dotsb+P_i) \label{Bneq3}
\end{align}
where the last relation follows from:
\[ \pi(x_i^*x_i)\zeta_{e_{jj}^{m_j}\cdots} = \casesif{(1-q^{m_i+2})\zeta_{e_{ii}^{m_i}\cdots}}{i=j}{(1-q)\zeta_{e_{jj}^{m_j}\cdots}}{j<i} \]
and
\[ \pi(x_ix_i^*)\zeta_{e_{ii}^{m_i}\cdots} = (1-q^{m_i+1})\zeta_{e_{ii}^{m_i}\cdots} \]
where these operators are zero on all other paths.

Intuitively, $x_ix_j=0$ for $i<j$ since $\pi(x_j)$ adds an edge starting from $v_j$, if possible. But then we cannot add on an edge starting from $v_i$, as there is no path from $v_i$ to $v_j$ if $i<j$. Intuitively, $x_i^*x_j=0$ for $i\ne j$ since $\pi(x_j)$ adds an edge starting from $v_j$, if possible. But $x_i^*$ subtracts an edge starting from $v_i$ if possible, so the composition is zero for $i\ne j$.

\begin{lemma}
$C^*(E_n)=C^*(z_1,\dotsc,z_n)$.
\end{lemma}

\begin{proof}
We will show that $C^*(z_1,\dotsc,z_n)=C^*(S_1,\dotsc,S_n)$, from which the result follows. Now $\abs{z_n}$ is invertible, so $z_n$ has a unique polar decomposition with phase $\phase{z_n}$ which we can readily show is $S_n$ by computing in the representation. For $i\le n-1$, $\abs{z_i}$ is invertible in the unital \cstsubalg\ $(P_0+\dotsb+P_i)C^*(E)(P_0+\dotsb+P_i)$. Write $\widetilde{z_i}$ for the inverse of $\abs{z_i}$ in the \cstsubalg. Then $\widetilde{z_i}\abs{z_i}=\abs{z_i}\widetilde{z_i}=P_0+\dotsb+P_i$, and $z_i$ is polar decomposable with phase $z_i\widetilde{z_i}$. Calculating using the representation, we may show that this phase is $S_i$. So we can recover $S_i$ from $x_i$, $i=1,\dotsc,n$, and so $C^*(z_1,\dotsc,z_n)=C^*(S_1,\dotsc,S_n)=C^*(E_n)$.
\end{proof}

\begin{proof}[Proof of Theorem \ref{thrmEn}]
We want to find homomorphisms
\[ \morphgraph{C^*(E_n)}{\qquad\quad\; C^*(x_1,\dotsc,x_n)} \]
Define
\[ \psi(x_i) := z_i, \qquad i=1,\dotsc,n. \]
Define $\psi$ by the results of Lemma \ref{lemBna}:
\begin{align*}
\phi(P_0)    &:= \alpha_1^*\alpha_1-\alpha_1\alpha_1^* & \phi(P_j) &:= \alpha_j\alpha_j^* \quad\text{for }j\ge 1 \\
\phi(S_{i0}) &:= \alpha_i(\alpha_1^*\alpha_1-\alpha_1\alpha_1^*) & \phi(S_{ij}) &:= \alpha_i\alpha_j\alpha_j^* \quad\text{for }j\ge 1.
\end{align*}
These images are projections and partial isometries, as required. We must show that these maps preserve the relations. Equations \eqref{Bneq1}-\eqref{Bneq3} show that $\psi$ preserves the relations for $C^*(x_1,\dotsc,x_n)$, if we can show that $\psi(Q_i)=P_0+\dotsb+P_i$. This is trivial for $i=n$ since $\psi(1)=1$, and for the other cases,
\[ \psi(Q_{i-1}) = \psi(\alpha_i^*\alpha_i-\alpha_i\alpha_i^*) = S_i^*S_i-S_iS_i^* = P_0+\dotsb+P_{i-1} \]
since $\psi(\alpha_i)=S_i$, because phase is preserved by a homomorphism.

To show $\phi$ extends to a homomorphism, we must show it preserves the Cuntz-Krieger relations. Firstly $\phi(S_i)=\alpha_i$ since
\begin{multline*}
\phi(S_i) = \phi(S_{i0}+S_{i1}+\dotsb+S_{ii}) = \alpha_i(\alpha_1^*\alpha_1-\alpha_1\alpha_1^*+\alpha_1\alpha_1^*+\dotsb+\alpha_i\alpha_i^*) \\
= \alpha_i(Q_1-R_1+R_1+\dotsb+R_i) = \alpha_iQ_i = \alpha_i.
\end{multline*}
We also have $\phi(P_i)=R_i$ for $i=1,\dotsc,n$, and $\phi(P_0)=Q_1-R_1$. Also $\phi(P_0+\dotsb+P_i)=Q_i$. Then:
\subsubsection*{(i) $P_iP_j=0$ for $i\ne j$}
Since $P_iP_j=P_jP_i$ we can assume $i<j$ without loss of generality.
  \begin{itemize}
  \item Case $0=i<j$. Then $\phi(P_0)\phi(P_j) = (Q_1-Q_1)R_j = Q_1R_j-R_1R_j$. But
  \[ Q_1R_j-R_1R_j = Q_1R_1-R_1R_1 = R_1-R_1 = 0 \]
  if $j=1$, and $Q_1R_j-R_1R_j = 0$ for $j\ge 2$.
  \item Case $0<i<j$. Then $\phi(P_i)\phi(P_j) = R_iR_j = 0$.
  \end{itemize}
\subsubsection*{(ii) $S_{ij}^*S_{ij}=P_j$}
Note $\phi(S_{ij}) = \phi(S_iP_j) = \alpha_i\phi(P_j)$.
  \begin{itemize}
  \item Case $j=0$. Then $\phi(S_{i0}) = \alpha_i(Q_1-R_1)$, so
  \begin{align*}
  \phi(S_{i0})^*&\phi(S_{i0}) = (Q_1-R_1)\alpha_i^*\alpha_i(Q_1-R_1) = (Q_1-R_1)Q_i(Q_1-R_1) \\
  &= (Q_1-R_1)^2 = Q_1-R_1 = \phi(P_0).
  \end{align*}
  \item Case $j\ge 1$. Then $\phi(S_{ij}) = \alpha_iR_j$, so
  \[ \phi(S_{ij})^*\phi(S_{ij}) = R_j\alpha_i^*\alpha_iR_j = R_jQ_iR_j = R_j^2 = R_j = \phi(P_j). \]
  \end{itemize}
\subsubsection*{(iii) $P_i = \sum_{j=0}^i S_{ij}S_{ij}^* = S_{i0}S_{i0}^*+\sum_{j=1}^i S_{ij}S_{ij}^*$ for $i\ge 1$}
Then
\begin{align*}
\phi\Bigl(\sum_{j=0}^i S_{ij}S_{ij}^*\Bigr) &= \alpha_i(Q_1-R_1)\alpha_i^* + \sum_{j=1}^i\alpha_iR_j\alpha_i^* \\
&= \alpha_i(Q_1-R_1+R_1+\dotsb+R_i)\alpha_i^* \\
&= \alpha_iQ_i\alpha_i^* = \alpha_i\alpha_i^*\alpha_i\alpha_i^* = R_i^2 = R_i = \phi(P_i),
\end{align*}
since $Q_1-R_1+R_1+\dotsb+R_i = Q_1+R_2+\dotsb+R_i = Q_2+R_3+\dotsb+R_i = \cdots = Q_{i-1}+R_i = Q_i$.

Finally, we must show that $\phi$ and $\psi$ compose to give the identity maps. But
\[ (\psi\circ\phi)(S_i) = \psi(\alpha_i) = \psi(\Ph(x_i)) = \Ph(\psi(x_i)) = \Ph(z_i) = S_i, \]
since the phase is preserved by a homomorphism. Hence $\psi\circ\phi=\id$ on the Cuntz-Krieger generators, since these are determined by $S_1$ and $S_2$, and so $\psi\circ\phi=\id_{C^*(E_n)}$. To show $\phi\circ\psi=\id$, note firstly that $C^*(x_i)\cong C(\D_q)$, as shown in the proof that $C^*(x_1,x_2)\cong C^*(E_2)$. As argued there, this implies that the images under $\phi$ of the Cuntz-Krieger generators are nonzero, and that $(\phi\circ\psi)(x_i)=x_i$ since we can make computations in $C^*(E_n)$. Hence $\phi\circ\psi=id_{C^*(x_1,\dotsc,x_n)}$. Consequently $\phi$ and $\psi$ are both injective and surjective, and so $C^*(E_n)\cong C^*(x_1,\dotsc,x_n)$.
\end{proof}

\subsection{Representations}

\begin{thrm}
Every irreducible \strepn\ of $C(\B_q^{2n})\cong C^*(x_1,\dotsc,x_n)$ is unitarily equivalent to one of the following representations:
\begin{enumerate}
\item a representation $\pi$ defined on a Hilbert space $\Hi$ with orthonormal basis $\set{\zeta_\mu}$ indexed by paths of finite length in $E_n$ ending at $v_0$ by:
\[ \pi(x_i)\zeta_{e_{jj}^{m_j}\cdots} = \begin{cases}
  \sqrt{1-q}\zeta_{e_{ii}^0e_{ij}e_{jj}^{m_j}\cdots} & \text{if } j<i \\
  \sqrt{1-q^{m_i+2}}\zeta_{e_{ii}^{m_i+1}\cdots}     & \text{if } j=i \\
  0                                                  & \text{if } j>i,
\end{cases} \]
for $i=1,\dotsc,n$
\item a family of representations $\varepsilon_{k,\theta}$ indexed by $k=1,\dotsc,n-1$, and $\theta\in S^1$, on the Hilbert space $\Hi$ with orthonormal basis indexed by paths of finite length ending at $v_k$ and which do not contain any loops $e_{kk}$ defined by
\begin{align*}
\varepsilon_{k,\theta}(x_i) &= 0 \qquad\text{for }i<k \\
\varepsilon_{k,\theta}(x_k) &\colon \zeta_{v_k}\mapsto \theta\zeta_{v_k}, \text{ and }\zeta_\alpha\mapsto 0 \text{ if } \alpha\ne v_k \\
\varepsilon_{k,\theta}(x_i) &= \pi^{(n-k)}(x_{i-k}) \qquad\text{for } i>k,
\end{align*}
where $\pi^{(n-k)}$ is the map $\pi$ on $C^*(E_{n-k})$ defined in (i), after relabelling $v_i'=v_{i-k}$, $e_{ij}'=e_{i-k,j-k}$, 
\item a one-dimensional representation $\sigma_\theta$ where $\theta\in S^1$, defined by
\[ \sigma_\theta(x_n) = \theta, \qquad \sigma_\theta(x_i) = 0 \quad\text{for }i<n. \]
\end{enumerate}
Only the representation $\pi$ is faithful.
\end{thrm}

\begin{proof}
$C^*(x_1,\dotsc,x_n)\cong C^*(E_n)$. Hence the ideals and representations of $C^*(x_1,\dotsc,x_n)$ correspond exactly to the ideals and representations of $C^*(E_n)$, and so we can use the powerful tool of graph algebras to classify these representations. The set $\set{v_0}$ is hereditary and saturated, and the ideal $J_0$ generated by $P_0$ is isomorphic to the compacts. The quotient $C^*(E_n)/J_0$ is isomorphic to the \cstalg\ $C^*(E_n\backslash\set{v_0})$, so we have an exact sequence
\[ 0\to J_0\to C^*(E_n)\to C^*\Bigl( \dotsb
\xygraph{{v_3}="v_3"(:^{e_{33}}@(ul,ur)"v_3" :^{e_{32}}[r]{v_2}="v_2"( :^{e_{22}}@(ul,ur)"v_2" :^{e_{21}}[r]{v_1}="v_1" :^{e_{11}}@(ul,ur)"v_1"), :_{e_{31}}@/_0.7pc/"v_1"}
\Bigr)\to 0. \]
Hence $\Irr(C^*(E_n))=\Irr(J_0)\sqcup\Irr C^*\Bigl( \dotsb
\xygraph{{\V}="v_3"(:^{e_{33}}@(ul,ur)"v_3" :^{e_{32}}[r]{\V}="v_2"( :^{e_{22}}@(ul,ur)"v_2" :^{e_{21}}[r]{\V}="v_1" :^{e_{11}}@(ul,ur)"v_1"), :_{e_{31}}@/_0.7pc/"v_1"} \Bigr)$. There is only one irreducible representation of $J_0$ up to unitary equivalence, which is the representation $\pi$ used previously when extended to a representation on $C^*(E_n)$. Composing with $\psi$ gives the representation in the Theorem.

Now $C^*(E_n\backslash\set{v_0})$ is the quantum sphere $C(S_q^{2n-1})$. This \cstalg\ is not simple, since $\set{v_1}$ is hereditary and saturated. Let $J_1$ be the ideal generated by $P_1$ in this algebra. We have an exact sequence
\[ 0\to J_1\to C^*(E_n\backslash\set{v_0})\to C^*(E_n\backslash\set{v_0,v_1})\to 0. \]
Hence $\Irr(C^*(E_n\backslash\set{v_0}) = \Irr(J_1)\sqcup\Irr(C^*(E_n\backslash\set{v_0,v_1}))$. The ideal $J_1$ is isomorphic to $K\otimes C(S^1)$ by the following reasoning. Consider the set $\Lambda_n^1$ of all paths in $E_n$ which end at $v_1$ and do not contain $e_{11}$. Note that $\Lambda_n^1$ has a direct correspondence with $\Lambda_{n-1}$, the set of paths in $E_{n-1}$ ending at $v_0$. Then $\set{S_\alpha S_\beta^*\colon \alpha,\beta\in\Lambda_n^1}$ is a system of matrix units inside $J_1$. Hence their closed span is isomorphic to $\K$.

Also, $S_{11}\in J_1$ since $S_{11}=S_{11}P_1$. $S_{11}$ is a partial unitary with full spectrum, so $C^*(S_{11})\cong C(S^1)$. Since $S_{11}$ commutes with $S_\alpha S_\beta^*$ aforementioned, by the general theory of graph algebras we have $J_1\cong\K\otimes C(S^1)$. This ideal has $S^1$ of irreducible representations. Fix $\theta\in S^1\subset\C$. Then define a representation $\varepsilon_{1,\theta}\colon J_w\to\B(\Hi)$, where $\Hi$ is the Hilbert space with orthonormal basis indexed $\Lambda_n^i$, and extend it to a representation on the whole algebra:
\[ \varepsilon_{1,\theta}(S_{11})\zeta_{v_1} = \theta\zeta_{v_1}, \]
and $\varepsilon_{1,\theta}(S_{11})$ maps all other paths to zero;
\[ \varepsilon_{1,\theta}(S_{ij})\zeta_\alpha = \casesother{\zeta_{e_{ij}\alpha}}{j=s(\alpha)}{0} \]
for $i\ge j\ge 1$ and $i\ne 1$.
Then
\[ \varepsilon_{1,\theta}(S_1)\zeta_{v_1} = \theta\zeta_{v_1}, \]
and zero otherwise. Also
\[ \varepsilon_{1,\theta}(S_i)\zeta_\alpha = \zeta_{e_{i,s(\alpha)}\alpha} \]
or zero otherwise, for $i\ge 1$. The yields the representation stated in the Theorem.

So we consider the \cstalg\ $C^*(E_n\backslash\set{v_0,v_1})$. Then $\set{v_2}$ is hereditary and saturated in the graph, so let $J_2$ be the ideal generated by $v_2$ so we have an exact sequence
\[ 0\to J_2\to C^*(E_n\backslash\set{v_0,v_1})\to C^*(E_n\backslash\set{v_0,v_1,v_2}). \]
Then $J_2\cong\K\otimes C(S^1)$ and we have a representation $\varepsilon_{2,\theta}$ similarly, and so on. We keep quotienting out vertices one by one, until we are left with vertices $\set{v_n,v_{n-1}}$. This is the graph for $C(S_q^3)$, and we have a representation $\varepsilon_{n-1,\theta}$.

Finally, $\set{v_{n-1}}$ is hereditary and saturated in the graph $E_n\backslash\set{v_0,\dotsc,v_{n-2}}$ and the quotient algebra is $C(S^1)$, which has characters $\sigma_\theta$ defined by $\sigma_\theta(e_{nn})=\theta$ for $\theta\in S^1$, so $\sigma_\theta(S_n)=\theta$ and $\sigma_\theta(S_i)=0$ for $i<n$. Hence $\sigma_\theta(x_n)=\theta$ and $\sigma_\theta(x_i)=0$ for $i<n$.

These are all the irreducible representations of $C^*(x_1,\dotsc,x_n)$ up to unitary equivalence.
\end{proof}

So the irreducible representations of $C(B_q^{2n})\cong C^*(x_1,\dotsc,x_n)$ are parameterised by a point and $n$ circles:
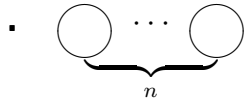
\[ \centerdot \qquad \underbrace{\text{\circle{20} } \quad\dotsb\quad \text{ \circle{20}}}_n \]
Only the representation $\pi$ is faithful.

\section{Appendix 1 - \texorpdfstring{$C^*(S_1,S_2)$}{C(S1,S2)} as a Universal Algebra}

In Section \ref{secBq4} we showed that $S_1$ and $S_2$ generate $C^*(E_2)$, where $E_2$ is the graph for the quantum $4$-ball $C(B_q^4)$:
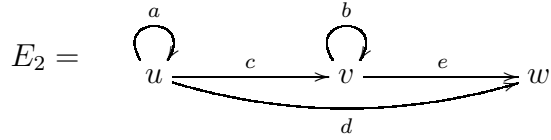
\[ E_2 = \quad \graphEfour \]
There we simply regarded $S_1$ and $S_2$ as lying inside $C^*(E_2)$. However we can also define $C^*(S_1,S_2)$ as a universal \cstalg\ generated by $S_1$ and $S_2$ subject to relations between them. We can show the following relations for $S_1$, $S_2$ in $C^*(E_2)$:
\begin{align*}
S_2^*S_2&=1 & S_2S_2^*&=P_u \\
S_1^*S_1&=P_v+P_w & S_1S_1^*&=P_v \\
S_1S_2&=0 & S_1^*S_2&=0.
\end{align*}
In fact these relations are sufficient to characterise $C^*(S_1,S_2)$ as a universal algebra.
\begin{defn}
Let $C^*(T_1,T_2)$ be the universal \cstalg\ generated by $T_1$, $T_2$ s.t.
\begin{align*}
T_2^*T_2&=1 & T_2T_2^*&<1 \\
T_1^*T_1&=1-T_2T_2^* & T_1T_1^*&<T_1^*T_1 \\
T_1T_2&=0 & T_1^*T_2&=0.
\end{align*}
\end{defn}

The relations imply $T_2$ is a proper isometry, hence $T_2T_2^*$ is a projection, and so $T_1^*T_1$ is a nonzero projection, and so $T_1$ is a partial isometry. Furthermore, $T_2T_2^*$ is nonzero, since $T_2T_2^*=0 \imp T_2T_2^*T_2=0 \imp T_2 = 0$, and $T_2\ne 0$ since there is a surjective homomorphism onto the \cstalg\ $C^*(E_2)=C^*(S_1,S_2)$ by universality, in which the image $S_2$ of $T_2$ is nonzero. We may show in a similar manner that $T_1T_1^*$ is also nonzero. 

\begin{thrm}
$C^*(S_1,S_2)\cong C^*(T_1,T_2)$.
\end{thrm}

\begin{proof}
We want to find homomorphisms
\[ \morphgraph{C^*(E_2)}{\quad\;\; C^*(T_1,T_2)} \]
Define
\[ \psi(T_i):=S_i, \qquad i=1,2. \]
From the results of Lemma \ref{lemB4a} we define
\[ \phi(P_u):=T_2T_2^* \qquad \phi(P_v):=T_1T_1^* \quad \dotsb \quad \phi(S_e):=T_1(1-T_1T_1^*), \]
replacing each $S_i$ in the Lemma with $T_i$. Then $\psi$ extends to a homomorphism because the relations between $T_1$ and $T_2$ are derived from the relations between $S_1$ and $S_2$. To show $\phi$ extends to a homomorphism, we must show that it preserves the Cuntz-Krieger relations, for which we show some representative examples. Now $P_vP_w=0$, and
\[ \phi(P_v)\phi(P_w) = T_1T_1^*(T_1^*T_1-T_1T_1^*) = T_1T_1^*-T_1T_1^* = 0, \]
since $T_1T_1^*<T_1^*T_1$ and these are projections. Also $S_c^*S_c=P_v$, and
\[ \phi(S_c)^*\phi(S_c) = T_1T_1^*T_2^*T_2T_1T_1^* = T_1T_1^*T_1T_1^* = T_1T_1^* = \phi(P_v), \]
since $T_2^*T_2=1$ and $T_1T_1^*$ is a projection. Furthermore, $S_bS_b^*+S_eS_e^*=P_v$ and
\begin{multline*}
\phi(S_b)\phi(S_b)^*+\phi(S_e)\phi(S_e)^* = T_1^2T_1^*T_1(T_1^*)^2+T_1(1-T_1T_1^*)^2T_1^* \\
= T_1^2(T_1^*)^2+T_1(1-T_1T_1^*)T_1^* = T_1(T_1T_1^*+1-T_1T_1^*)T_1^* = T_1T_1^* = \phi(P_v).
\end{multline*}
Now $\phi(S_i)=T_i$ for $i=1,2$, since
\[ \phi(S_1)=\phi(S_b+S_e)=T_1(T_1T_1^*+1-T_1T_1^*) = T_1, \]
and
\[ \phi(S_2)=\phi(S_a+S_c+S_d)=T_2(T_2T_2^*+T_1T_1^*+T_1^*T_1-T_1T_1^*) = T_2, \]
by the relation $T_1^*T_1=1-T_2T_2^*$. Hence $(\psi\circ\phi)(S_i)=S_i$ and so $\psi\circ\phi=\id$ on the Cuntz-Krieger generators as well, so $\psi\circ\phi=\id_{C^*(E_2)}$. Similarly $\phi\circ\psi=\id_{C^*(T_1,T_2)}$ and so these maps are isomorphisms and $C^*(S_1,S_2)\cong C^*(E_2)\cong C^*(T_1,T_2)$.
\end{proof}

Just as we had three isomorphic \cstalg s describing the quantum disc - $C(\D_q)$, $\T=C^*(S)$ and $C^*(E_1)$ - we have three isomorphic \cstalg s describing the quantum $4$-ball:
\[ C^*(x_1,x_2) \quad\cong\quad C^*(T_1,T_2) \quad\cong\quad C^*\Bigl(\graphEfourS\Bigr) \]

\section{Appendix 2 - \texorpdfstring{$C^*(S_1,\dotsc,S_n)$}{C*(S1,...,Sn)} as a Universal Algebra}

We can also regard $C^*(S_1,\dotsc,S_n)$ as a universal algebra generated by $S_1,\dotsc,S_n$ with certain relations between them.

\begin{defn}
Let $C^*(T_1,\dotsc,T_n)$ be the universal \cstalg\ generated by $T_1,\dotsc,T_n$ s.t.
\begin{align*}
T_iT_j &= 0 \qquad\text{if }i<j \\
T_i^*T_j &= 0 \qquad\text{if }i\ne j \\
T_n^*T_n &= 1 \\
T_{i-1}^*T_{i-1} &= T_i^*T_i-T_iT_i^*, \qquad i=2,\dotsc,n \\
T_1T_1^* &< T_1^*T_1.
\end{align*}
\end{defn}

The relations imply that the $T_i$ are partial isometries, and that $T_n$ is a proper isometry. Also we must have $T_iT_i^*<T_i^*T_i$ since $\pi\colon C^*(E_2)\to\B(\Hi)$ is a representation of these relations in which $T_i\ne 0$.

\begin{thrm}
$C^*(S_1,\dotsc,S_n)\cong C^*(T_1,\dotsc,T_n)$.
\end{thrm}

\begin{proof}
As before we may exhibit isomorphisms
\[ \morphgraph{C^*(E_n)}{\qquad\quad\; C^*(T_1,\dotsc,T_n)} \]
as before, where $\phi$ is defined via the results of Lemma \ref{lemBna} and $\psi(T_i):=S_i$ for $i=1,\dotsc,n$. We may then show that $\phi(S_i)=T_i$ for all $i$, and that the images of the generators are all nonzero. These maps preserve the relations. Clearly $\phi\circ\psi=\id$ and $\psi\circ\phi=\id$, and so we have isomorphism.
\end{proof}

Hence we have three isomorphic \cstalg s describing the quantum $2n$-ball:
\[ C^*(x_1,\dotsc,x_n) \quad\cong\quad C^*(T_1,\dotsc,T_n) \quad\cong\quad C^*(E_n) \]

\end{document}